\documentclass[11pt]{article}
\usepackage{geometry}                
\usepackage{graphicx}
\usepackage{amssymb}
\usepackage{epstopdf}
\usepackage[T1]{fontenc}
\usepackage[ngerman, english]{babel}
\usepackage{mathtools}
\usepackage{bbm}
\usepackage{mathrsfs}
\usepackage{amsmath,amscd}
\usepackage{tikz}
\usepackage{tikz-cd}
\usepackage{babel}
\usetikzlibrary{arrows, matrix}
\usetikzlibrary{cd}
\usepackage[arrow, matrix, curve]{xy}

\newcommand{\qed}{\hfill \ensuremath{\Box}}
\newenvironment{proof}{\vspace{1ex}\noindent{\it Proof.}\hspace{0.5em}}
	{\hfill\qed\vspace{1ex}}

\newtheorem{theorem}{Theorem}[section]
\newtheorem{lemma}[theorem]{Lemma}
\newtheorem{proposition}[theorem]{Proposition}
\newtheorem{corollary}[theorem]{Corollary}

\newtheorem{definition}[theorem]{Definition}

\DeclareMathOperator{\Gal}{\operatorname{Gal}}
\DeclareMathOperator{\Q}{\mathbf{Q}}

\DeclareMathOperator{\Z}{\mathbf{Z}}

\DeclareMathOperator{\N}{\mathbf{N}}

\DeclareMathOperator{\im}{\mathrm{im}}
\DeclareMathOperator{\Spec}{\operatorname{Spec}}
\DeclareMathOperator{\Proj}{\operatorname{Proj}}

\DeclareMathOperator{\Hom}{\operatorname{Hom}}

\DeclareMathOperator{\Aut}{\operatorname{Aut}}

\DeclareMathOperator{\Og}{\mathcal{O}}

\DeclareMathOperator{\Pic}{\mathrm{Pic}}

\DeclareMathOperator{\rk}{\mathrm{rk}}

\DeclareMathOperator{\et}{\acute{\mathrm{e}}{\mathrm{t}}}
\DeclareMathOperator{\Gm}{\mathbf{G}_m}

\DeclareMathOperator{\ev}{\mathrm{ev}}

\title{Degeneration of Kummer surfaces}
\author{Otto Overkamp}
\date{}
\begin{document}
\maketitle
{\abstract{We prove that a Kummer surface defined over a complete strictly Henselian discretely valued field $K$ of residue characteristic different from 2 admits a strict Kulikov model after finite base change. The Kulikov models we construct will be schemes, so our results imply that the semistable reduction conjecture is true for Kummer surfaces in this setup, even in the category of schemes. Our construction of Kulikov models is closely related to an earlier construction of K\"unnemann, which produces semistable models of Abelian varieties. It is well-known that the special fibre of a strict Kulikov model belongs to one of three types, and we shall prove that the type of the special fibre of a strict Kulikov model of a Kummer surface and the toric rank of a corresponding Abelian surface are determined by each other. We also study the relationship between this invariant and the Galois representation on the second $\ell$-adic cohomology of the Kummer surface. Finally, we apply our results, together with earlier work of Halle-Nicaise, to give a proof of the monodromy conjecture for Kummer surfaces in equal characteristic zero.}}
\tableofcontents
\section{Introduction}
Let $\Og_K$ be a complete discrete valuation ring with maximal ideal $\mathfrak{m}_K$ and algebraically closed residue field $k$ of characteristic $p\not=2.$ Let $K$ be the field of fractions of $\Og_K.$ Let $A$ be an Abelian surface over $K$, and let $X$ be the associated Kummer surface. In other words, let $X$ be the quotient the blow-up of $A$ in $A[2]$ by the action of $\Z/2\Z$ given by the involution $[-1]$. Then $X$ is a smooth surface over $K$ which turns out to be a K3 surface. (See \cite{B}, Chapter 10, Theorem 10.6 for this particular result, and \cite{Huy} for a more general introduction to K3 surfaces). This surface is called the \it Kummer surface associated with \rm $A$. The purpose of the present paper is to study the existence and properties of \it (strict) Kulikov models \rm of the Kummer surface $X$. By definition, a strict Kulikov model of $X$ (or a more general K3 surface) is a regular algebraic space $\mathscr{X}$ which is proper and flat over $\Og_K$, is a model of $X$, which has the property that its special fibre is a reduced divisor with strict normal crossings (see the definition below) on $\mathscr{X}$, and such that the relative dualizing sheaf $\omega_{\mathscr{X}/\Og_K}$ is trivial. Models of this kind were originally studied by Kulikov (\cite{Ku}) and Persson-Pinkham (\cite{Per}, \cite{PerPink}) in the context of complex-analytic geometry. It can be shown that the special fibre of such a model, if it exists, is a \it combinatorial K3 surface \rm (which we shall define later), and that every combinatorial K3 surface belongs to one of three types. One can also show that the type only depends upon the $\ell$-adic Galois representation $H^2_{\et}(X_{\overline{K}}, \Q_\ell)$, and hence only upon the generic fibre of $\mathscr{X}$. Our first main result is
\begin{theorem} \rm(Theorem \ref{Kummersemistabletheorem}) \it
Let $A$ be an Abelian surface over $K$ and let $X$ be the associated Kummer surface. Then there exists a finite extension $F$ of $K$ such that there exists a strict Kulikov model $\mathscr{X}\to \Spec\Og_F$ of $X_F:=X\times_K\Spec F.$ Moreover, the Kulikov model $\mathscr{X}$ which we construct is a scheme. 
\end{theorem}
This implies in particular that Kummer surfaces admit potentially semistable reduction in the category of schemes. For the proof of this result it is clearly harmless to assume from the beginning that $A$ has semiabelian reduction, and that the $K$-group scheme $A[2]$ is constant. The strategy for proving this result is as follows: First, we shall recall some of the results from K\"unnemann's celebrated paper \cite{Kü}. Among other things, K\"unnemann shows that there exists a finite extension $F$ of $K$ and a regular projective model $P$ of $A_F$ which has strict semistable reduction, and such that the involution $[-1]$ extends uniquely to $P$. Furthermore, the model $P$ of $A_F$ will contain the Néron model $\mathscr{A}_F$ of $A_F$ over $\Spec\Og_F$. The main step of our proof will consist in showing that the fixed locus of the involution of $P$ extending $[-1]$ coincides with the étale group scheme $\mathscr{A}_F[2]\subseteq P.$ From this it will follow that the singularities of the quotient of $P$ by the involution are mild enough to be resolved by blowing up a regular centre once, and that this blow-up will be our desired semistable model $\mathscr{X}$ of $X_F$. We shall also see that the model $P$ of $A_F$ is a strict Kulikov model of $A$. From this, we shall deduce that the model $\mathscr{X}$ of $X_F$ is, in fact, a strict Kulikov model. We shall then go on to studying the relationship between the degeneration behaviour of $A$ and that of $X$. In particular, we shall see that the type of the strict Kulikov model of $X_F$ can be read off from the toric rank of $A_F$ (which is a numerical invariant that will be introduced later). We shall also see that there is a close relationship between the dual complexes of the strict Kulikov models of $A_F$ and $X_F$. Kulikov models have been applied to the study of motivic zeta functions of K3-surfaces and the monodromy conjecture (\cite{HN3}, Definition 2.3.5), and our results will provide a new class of K3 surfaces which satisfy the monodromy property.\\
\\
$\mathbf{Remark.}$ Our first main theorem strengthens the following previously known result: Suppose that the residue characteristic of $K$ is at least 5 and that $X$ is a K3 surface over $K$ with Picard rank $12\leq \rho\leq 20.$ This applies in particular to Kummer surfaces, at least after a finite extension of the ground field. Then there exists a finite extension $L/K$ and a proper algebraic space $\mathscr{Y}\to \Spec \Og_K$ which is a strict Kulikov model of $X_L$ (see \cite{LM}, Proposition 3.1 together with \cite{Ito}, Propositions 2.3 and 2.4). In general, these algebraic spaces cannot be guaranteed to be schemes (see \cite{LM}, proof of Proposition 3.1), so this does not imply potential semistable reduction for Kummer surfaces in the category of schemes. Furthermore, the proof of this result in \it loc. cit. \rm relies on the (semistable) minimal model program and provides no control over the Kulikov models' special fibre. On the other hand, our construction is completely explicit, and also allows us to deal with Galois-equivariant (and non-strict) Kulikov models, which are essential for our applications to motivic Zeta functions. Such models do not seem to be accessible using the current methods. Furthermore, our method also applies in the case $p=3.$ The author is grateful to Professor C. Liedtke for bringing the paper \cite{Ito} to his attention.

\section{Preliminaries} \label{Prelimsection}
To avoid confusion, let us recall the definition of the notions of \it semistable reduction, divisor with normal crossings, \rm and \it divisor with strict normal crossings \rm which we shall use in this manuscript:
\begin{definition}
Let $S$ be a Noetherian scheme and let $D$ be an effective Cartier divisor on $S$. Let $D_1,..., D_r$ be the irreducible components of $D$ endowed with the structure of reduced closed subschemes of $S$. For each subset $J\subseteq \{1,...,r\},$ denote by $D_J$ the scheme-theoretic intersection $\cap_{j\in J} D_j,$ defining $D_{\emptyset}:=S.$\\
(i) We say that the divisor $D$ has \rm strict normal crossings \it if $D$ is reduced, for each point $s$ of $S$ contained in $D$, the local ring $\Og_{S,s}$ is regular, and if for each $\emptyset\not=J\subseteq \{1,..., r\}$, the scheme $D_J$ is regular and of codimension $\#J$ in $S$. \\
(ii) We say that the divisor $D$ has \rm normal crossings \it if $D$ has strict normal crossings locally in the étale topology.\\
(iii) Suppose $R$ is a discrete valuation ring with fraction field $L$ and perfect residue field. Let $X$ be a proper, smooth algebraic variety over $K$, and suppose that $\mathscr{X}\to \Spec R$ is a model of $X$. In particular, we fix an isomorphism between $X$ and $\mathscr{X}\times_R\Spec K$. We say that $\mathscr{X}$ is a \rm strictly semistable model \it of $X$ if it is flat, regular, and proper over $\Og_K$, and if the special fibre of $\mathscr{X}$ is a divisor with strict normal crossings on $\mathscr{X}.$ \\
(iv) We say that $\mathscr{X}$ is a \rm semistable model \it of $X$ if the special fibre of $\mathscr{X}$ has normal crossings, i.e., if $\mathscr{X}$ looks as in (iii) locally in the étale topology.
\end{definition}
This definition is the same as \cite{Kü}, (1.9). Observe that a divisor $D$ on $S$ which has normal crossings and whose irreducible components are regular has strict normal crossings. In particular, we can adapt this definition to the case where the model $\mathscr{X}$ is an algebraic space rather than a scheme: We say that $\mathscr{X}$ has semistable reduction if it admits an atlas whose special fibre is a divisor with normal crossings as in (ii), and we say that $\mathscr{X}$ has strictly semistable reduction if, in addition, the irreducible components of the special fibre of $\mathscr{X}$ are smooth over the residue field. Furthermore, we shall say that $X$ has \it strict semistable reduction \rm (resp. \it semistable reduction\rm) if $X$ admits a strictly semistable model (resp. a semistable model) which is a  \it proper scheme \rm over $R$.\\
In this chapter, we shall recall some basic material, mainly from \cite{Kü} and \cite{FC} (throughout, we shall follow the notation of \cite{Kü}).
\subsection{Projective models of Abelian surfaces after K\"unnemann}
\subsubsection{Various categories}
This section is devoted to recalling some basic material from \cite{Kü}, and adapting it for our purposes. For more details, the reader should consult \cite{Kü}. Let $G$ be a semiabelian scheme over $\Og_K$. By definition, this means that $G$ is a smooth separated group scheme of finite type over $\Og_K$ whose geometric fibres are extensions of Abelian varieties by algebraic tori. We shall always assume that $G$ is a model of our Abelian surface $A$. First recall the \it Raynaud extension \rm 
\begin{align}0\to T\to \tilde{G}\overset{\pi}{\to} E\to 0\label{Raynaudext}\end{align} associated with $G$, whose precise construction is explained in \cite{Kü}, 2.1. Here, $T$ is an algebraic torus, $E$ an Abelian scheme, and $\tilde{G}$ a semiabelian scheme over $\Og_K.$ Note that, in order to construct this extension, we need to choose a line bundle $\mathscr{L}$ on $G$ whose restriction $\mathscr{L}_\eta$ to $A=G_\eta$ is ample. The extension itself, however, is independent of the choice of $\mathscr{L}.$ The group scheme $\tilde{G}$ has the property that the formal completions of $\tilde{G}$ and of $G$ with respect to their special fibres are canonically isomorphic. The line bundle $\mathscr{L}$ induces a line bundle $\tilde{\mathscr{L}}$ on $\tilde{G}.$ From now on, we shall assume that all our line bundles have cubical structures (\cite{Kü}, (1.7)). Since the base scheme $\Spec \Og_K$ over which we are working is normal, choosing a cubical structure on a line bundle over a group scheme is equivalent to choosing a rigidification of this line bundle along the identity section.\\
We shall use the categories $\mathrm{DEG}^{\mathrm{split}}_{\mathrm{ample}}$ and $\mathrm{DD}^{\mathrm{split}}_{\mathrm{ample}}$ from \cite{Kü}. We refer the reader to \cite{Kü} for the precise definitions of these categories. Objects of the category $\mathrm{DEG}^{\mathrm{split}}_{\mathrm{ample}}$ of split ample degenerations are triples $(G, \mathscr{L}, \mathscr{M}),$ where $G\to \Spec\Og_K$ is a semiabelian scheme over $\Og_K,$ $\mathscr{L}$ a cubical invertible sheaf on $G$ with ample restriction to $A=G_\eta$, and $\mathscr{M}$ a cubical ample invertible sheaf on $E$ such that $\tilde{\mathscr{L}}=\pi^\ast\mathscr{M}.$ We shall also use the category $\mathrm{DD}^{\mathrm{split}}_{\mathrm{ample}}$ of split ample degeneration data, the objects of which are tuples
$$(E, X, Y, \phi, c, c^t, \tilde{G}, \iota, \tau, \tilde{\mathscr{L}}, \mathscr{M}, \lambda_E, \psi, a, b).$$
In this notation, $E$ stands for an Abelian scheme over $\Og_K.$ Furthermore, $X$ and $Y$ denote free Abelian groups of the same finite rank $r$, and $\phi\colon Y\to X$ is an injective homomorphism. Let $T$ be the torus $\Hom(X, \Gm).$ Then there is a canonical isomorphism $X^\ast(T)=X.$ Similarly, define $T'$ to be $\Hom(Y,\Gm).$ Next, $c$ and $c^t$ denote homomorphisms $c\colon X\to E^\vee(\Og_K)$ and $c^t\colon Y\to E(\Og_K).$ The morphism $c$ encodes an extension 
$$0\to T\to \tilde{G}\overset{\pi}{\to} E\to 0,$$ for some semiabelian scheme $\tilde{G}$ over $\Og_K$. Now $\iota$ is a homomorphism $\iota\colon Y\to \tilde{G}(K)$ such that $\pi\circ\iota=c^t.$ This $\iota$ is determined by a unique trivialization $\tau\colon \boldsymbol{1}_{(X\times Y)_\eta}\to (c\times c^t)^\ast \mathscr{P}_{E_\eta}^{-1}$ of biextensions. Here $\mathscr{P}_E$ denotes a rigidified Poincaré bundle on $E\times_{\Og_K} E^\vee.$  Next we choose a cubical ample invertible sheaf $\mathscr{M}$ on $E$ and put $\tilde{\mathscr{L}}:=\pi^\ast\mathscr{M}.$ We let $\lambda_E$ be the polarization $E\to E^\vee$ associated with $\mathscr{M}$, and let $\psi\colon\boldsymbol{1}_{Y_\eta} \to \iota^\ast \tilde{\mathscr{L}}_\eta$ be a trivialization of $\Gm$-torsors. Finally, we let $a$ and $b$ be a function $a\colon Y\to \Z$ and a bilinear pairing $b\colon Y\times X\to \Z$ which are determined by $\psi$ and $\tau$, respectively. These data are subject to several compatibility requirements which we have not mentioned at this point; the missing details can be found in \cite{Kü}, (2.2).

There is a natural functor
$$F\colon \mathrm{DEG}^{\mathrm{split}}_{\mathrm{ample}}\to\mathrm{DD}^{\mathrm{split}}_{\mathrm{ample}},$$ which turns out to be an equivalence of categories (\cite{Kü}, (2.8)). If $(G, \mathscr{L}, \mathscr{M})$ is an object of $\mathrm{DEG}^{\mathrm{split}}_{\mathrm{ample}}$, then $E$, $\tilde{G}$, and $\tilde{\mathscr{L}}$ (which appear in $F((G, \mathscr{L}, \mathscr{M}))$) come from the Raynaud extension described at the beginning of this paragraph. Furthermore, $X$ is defined to be $X^\ast(T),$ and $c^t$ encodes the Raynaud extension. 

 There is one further category which will be important in what follows, namely the category $\mathcal{C}$. Objects of this category are tuples
$(X,Y,\phi, a,b),$ where $X$ and $Y$ are free Abelian groups of the same finite rank, $\phi\colon Y\to X$ is an injective morphism, $a\colon Y\to \Z$ is a function with $a(0)=0,$ and $b\colon Y\times X\to \Z$ a bilinear pairing such that $b(-, \phi(-))$ is symmetric, positive definite, and satisfies
$$a(y+y')-a(y)-a(y')=b(y, \phi(y')).$$ 
A morphism $(X', Y', \phi', a', b')\to (X,Y,\phi,a,b)$ is a pair of morphisms $h_X\colon X\to X'$ and $h_Y\colon Y'\to Y$ such that $a'=a\circ h_Y,$ $b'(-, h_X(-))=b(h_Y(-),-),$ and $\phi'=h_X\circ \phi\circ h_Y.$ It follows from \cite{Kü}, (2.4) that the association 
$$(E, X, Y, \phi, c, c^t, \tilde{G}, \iota, \tau, \tilde{\mathscr{L}}, \mathscr{M}, \lambda_E, \psi, a, b) \mapsto (X,Y, \phi, a,b)$$ defines a functor
$$\mathrm{For}\colon \mathrm{DD}^{\mathrm{split}}_{\mathrm{ample}} \to \mathcal{C}.$$ One should note that the categories $\mathrm{DEG}^{\mathrm{split}}_{\mathrm{ample}}$ and $\mathrm{DD}^{\mathrm{split}}_{\mathrm{ample}}$ depend on the ground field $K$ whereas $\mathcal{C}$ is independent of the ground field. Just for the moment, we shall include the ground field in our notation, in order to state some results about the behaviour of the functor $\mathrm{For}$ under base change. Thereafter, we shall omit any reference to the ground field, as we have done before. Let $F/K$ be a finite extension of ramification index $\nu$. We obtain a base change functor
$$-\times_K\Spec F\colon \mathrm{DD}^{\mathrm{split}}_{\mathrm{ample},K}\to \mathrm{DD}^{\mathrm{split}}_{\mathrm{ample},F},$$ and similarly for $\mathrm{DEG}^{\mathrm{split}}_{\mathrm{ample}}.$ For any object $D$ of $\mathrm{DEG}^{\mathrm{split}}_{\mathrm{ample},K}$, if $\mathrm{For}(D)=(X,Y, \phi, a, b)$, then $\mathrm{For}(D\times_K\Spec F)$ is canonically isomorphic to $(X,Y, \phi, \nu\cdot a, \nu\cdot b)$ (see \cite{Kü}, (2.9)).\\
Now let $G:=\mathscr{A}^0,$ where $\mathscr{A}$ is the Néron model of the Abelian surface $A$. It follows from \cite{FC}, Chapter I, Proposition 2.5 and \cite{Ra}, Chapter XI, Théorème 1.13 that the forgetful functor from $\mathrm{DEG}^{\mathrm{split}}_{\mathrm{ample}}$ into the category of semiabelian schemes over $\Og_K$ is essentially surjective. Suppose that we have an object $(G, \mathscr{L}, \mathscr{M})$ of $\mathrm{DEG}^{\mathrm{split}}_{\mathrm{ample}}$ with $G=\mathscr{A}^0.$
\subsubsection{Group actions on degenerations}
Now let $H$ be a finite group which acts (from the left) on $\Og_K$. By an \it action of $H$ on $(G, \mathscr{L}, \mathscr{M})$ over the action on $\Og_K$ \rm we mean a system of homomorphisms
$$h^\ast(G, \mathscr{L}, \mathscr{M})\to (G,\mathscr{L}, \mathscr{M}),$$ for each $h\in H,$ such that the obvious compatibilities are satisfied (see \cite{Kü}, (2.10)). If the finite group $H$ acts on $G$ in a way compatible with the action on $\Og_K$, we may replace $\mathscr{L}$ by $\bigotimes_{h\in H}h^\ast\mathscr{L}$ (and similarly for $\mathscr{M}$) and assume that $H$ acts on the object $(G,\mathscr{L}, \mathscr{M})$ of $\mathrm{DEG}^{\mathrm{split}}_{\mathrm{ample}}$ over the action on $\Og_K$. From now on, we shall always assume that a pre-image $(G, \mathscr{L}, \mathscr{M})$ of $\mathscr{A}^0$ has been chosen on which $H:=\{\mathrm{Id}, [-1]\}$ acts. Only in the last chapter will we be interested in actions of the group $H=\{\mathrm{Id}, [-1]\}\times \boldsymbol{\mu}_d$ on $(G, \mathscr{L}, \mathscr{M})\times_K\Spec K(d)$ over the action of $H$ (via the second factor) on $\Og_{K(d)}.$ Here we consider the unique extension $K(d)$ of $K$ of degree $d$ and identify its Galois group (which will act from the left) with $\boldsymbol{\mu}_d$ for $d\in \N,$ $p\nmid d.$ If $H=\{\mathrm{Id}, [-1]\},$ then we shall always assume that $H$ acts trivially on the base ring.
\subsubsection{K\"unnemann's construction}
Given an object $(X,Y, \phi, a, b)$ of $\mathcal{C}$ on which the finite group $H$ acts, we obtain an action (from the left) of $H$ on $Y$, and an action (from the right) of $H$ on $X$. Put $\Gamma:=Y\rtimes H.$ Then $\Gamma$ acts on the free $\Z$-module $X^\vee\oplus\Z$ as
$$S_{(y,h)}((l,s)):=(l\circ h+sb(y,-), s),$$ as in \cite{Kü}, p.181. In $X^\vee_{\mathbf{R}}\oplus\mathbf{R},$ we have the cone $\mathscr{C}:=(X^\vee_{\mathbf{R}}\times\mathbf{R}_{>0})\cup\{0\}.$ We shall consider a smooth $\Gamma$-admissible rational polyhedral cone decomposition $\{\sigma_\alpha\}_{\alpha\in I}$ which admits a $\Gamma$-admissible $\kappa$-twisted polarization function $\phi\colon \mathscr{C}=\bigcup_{\alpha\in I} \sigma_{\alpha}\to \mathbf{R}$ for some $\kappa\in \N,$ as in \cite{Kü}, p.181. The following result can be assembled from various Theorems and Propositions in \cite{Kü}:
\begin{theorem}
Let the finite group $H$ act on $\Og_K$ from the left. Let $(G, \mathscr{L}, \mathscr{M})\in \mathrm{DEG}^{\mathrm{split}}_{\mathrm{ample}}$ and assume that $H$ acts on this object over the action on $\Og_K$. Let $(X,Y,\phi, a,b):=\mathrm{For}(F((G, \mathscr{L}, \mathscr{M})))$ and suppose we have a smooth $\Gamma$-admissible rational polyhedral cone decomposition $\{\sigma_{\alpha}\}_{\alpha\in I}$ of $\mathscr{C}\subseteq X^\vee_{\mathbf{R}}\oplus\mathbf{R}$. Assume further that this cone decomposition has the following properties:\\
(a) There exists a $\Gamma$-admissible $\kappa$-twisted polarization function $\phi$ for this decomposition,\\
(b) The cone decomposition is \rm semistable \it in the sense that the primitive element of any one-dimensional cone contained in this decomposition is of the form $(\ell,1)$ for some $\ell\in X^\vee.$\\
(c) The cone $\sigma_T=\{0\}\times\mathbf{R}_{\geq0}$ is contained in this decomposition, and\\
(d) For all $y\in Y\backslash \{0\}$ and $\alpha\in I$, we have
$$\sigma_{\alpha}\cap S_{(y,\mathrm{Id})}(\sigma_{\alpha})=\{0\}.$$ 
Then there exists a regular irreducible scheme $P$ which is projective and flat over $\Og_K$ (depending on $\{\sigma_{\alpha}\}_{\alpha\in I}$) and a line bundle $\mathscr{L}_P$ (depending on the polarization function $\phi$) such that the following holds:\\
(i) There is an isomorphism $P\times_{\Og_K} \Spec K\to A$ (which we shall keep fixed from now on), and the canonical morphism
$$P^{\mathrm{sm}}\to \mathscr{A}$$ is an isomorphism.\\
(ii) The action of $H$ on $G=\mathscr{A}^0$ over the action of $H$ on $\Og_K$ extends uniquely to $P$, and the restriction of $\mathscr{L}_P$ to $G$ is isomorphic to $\mathscr{L}^{\otimes\kappa}$.\\
(iii) Let $I^+$ be the set of orbits $I^+:=(I\backslash\{\{0\}\})/Y$. Then the reduced special fibre of $P$ has a stratification indexed by $I^+$. This stratification is preserved by the action of $H$, and the induced action of $H$ on the set of strata is given by the action of $H$ on $I^+.$\\
(iv) The strata associated with one-dimensional cones are smooth over $k$.\\
(v) The special fibre of $P$ is a reduced divisor with strict normal crossings on $P$. \label{Pexistencetheorem}
\end{theorem}  
\begin{proof}
By \cite{Kü}, Theorem 3.5, there exists a scheme $P\to \Spec\Og_K$ (depending on the cone decomposition) which is regular as well as projective and flat over $\Og_K$, which contains $\mathscr{A}^0$ as an open subscheme, and which satisfies conditions (ii), (iii), and (v). The scheme $P$ is irreducible by \cite{FC}, Chapter III, Proposition 4.11. It follows in particular that $P$ is a model of $A$. By \cite{Kü}, (4.4), we know that $P$ contains the Néron model $\mathscr{A}$ of $A$ as an open subscheme. The open immersion $\mathscr{A}\to P$ must factor through $P^{\mathrm{sm}}$, and we have a canonical morphism $P^{\mathrm{sm}}\to \mathscr{A}$ from the universal property of the Néron model. Both compositions $\mathscr{A}\to P^{\mathrm{sm}}\to \mathscr{A}$ and $P^{\mathrm{sm}}\to \mathscr{A}\to P^{\mathrm{sm}}$ are equal to the identity because this holds generically. Hence part (i) follows. For part (iv), note that strata associated with one dimensional cones whose primitive element has the form $(\ell, 1)$ for some $\ell\in X^\vee$ are torsors for the group scheme $\mathscr{A}^0_k$. This follows from the construction of $P$ explained in \cite{Kü}. Since $\mathscr{A}^0_k$ is smooth over $k$, the claim follows.
\end{proof}\\
In general, we cannot expect a smooth rational polyhedral cone decomposition having properties (a),...,(d) to exist. We have, however, the following
\begin{proposition}
Let $(G, \mathscr{L}, \mathscr{M})\in\mathrm{DEG}^{\mathrm{split}}_{\mathrm{ample}}$, and assume that the finite group $H$ acts on this object (we assume in this Proposition that $H$ acts trivially on $\Og_K$). Let $(X,Y, \phi, a,b):=\mathrm{For}(F((G, \mathscr{L}, \mathscr{M}))).$ After replacing $K$ by a finite extension if necessary, there exists a smooth rational polyhedral cone decomposition $\{\sigma_{\alpha}\}_{\alpha\in I}$ which has the properties (a),..., (d) listed in Theorem \ref{Pexistencetheorem}. \label{semconeexistenceproposition}
\end{proposition}
\begin{proof}
It follows from \cite{Kü}, Theorem 4.7 that there is a smooth rational polyhedral cone decomposition $\{\sigma_{\alpha}\}_{\alpha\in I}$ which is $\Gamma$-admissible, admits a $\Gamma$-admissible $\kappa$-twisted polarization function (for some $\kappa\in \N$), and is semistable with respect to the integral structure given by $X^\vee\oplus (\nu\Z)$ for some positive integer $\nu.$ Furthermore, we know that this cone decomposition is constructed as a subdivision of one that has properties (c) and (d), which implies that $\{\sigma_{\alpha}\}_{\alpha \in I}$ will have those properties as well. Now choose any finite extension $L$ of $K$ with ramification index equal to $\nu.$ Letting $(X',Y', \phi', a', b'):=\mathrm{For}(F((G, \mathscr{L}, \mathscr{M})\times_K\Spec L)),$ we see that the map
\begin{align*}
(X')^\vee\oplus\Z&\to X^\vee\oplus\Z\\
(l,s)&\mapsto (l\circ h_X, \nu\cdot s)
\end{align*}
is $\Gamma$-equivariant, where $h_X\colon X\to X'$ is the canonical isomorphism (via the canonical isomorphism $Y'\to Y$ we can identify $Y'\rtimes H$ with $Y\rtimes H$, and we refer to both of them as $\Gamma$). Hence we obtain our desired cone decomposition by transport of structure. 
\end{proof}\\
Informally speaking, K\"unnemann's construction proceeds as follows: Given a $\Gamma$-admissible cone decomposition $\{\sigma_{\alpha}\}_{\alpha\in I}$ of $\mathscr{C}$ which admits a $1$-twisted $\Gamma$-admissible polarization function, we construct a scheme $Z=Z(\{\sigma_{\alpha}\}_{\alpha\in I})$, which is regular and locally of finite type over $\Og_K$, on which $T$ acts, and which contains $T$ as an open orbit (the action of $T$ on $Z$ extends the action of $T$ on itself by translation). This similar to the construction of a toric variety from a fan, and the existence of a cone decomposition satisfying our requirements follows after replacing $(G, \mathscr{L}, \mathscr{M})$ by $(G, \mathscr{L}^{\otimes n}, \mathscr{M}^{\otimes n})$ for $n$ sufficiently large. Since there is a $\Gamma$-admissible polarization function, we obtain a $T$-linearized ample line bundle $\mathscr{N}$ on $Z$. We let $\tilde{P}$ be the contracted product $\tilde{G}\times^{T}Z,$ where $\tilde{G}$ comes from the Raynaud extension associated with $A$ ($\tilde{P}$ will be a \it relatively complete model \rm for the object $F((G, \mathscr{L}^{\otimes n}, \mathscr{M}^{\otimes n}))$; see \cite{Kü}, Definition 2.12, or \cite{FC}, Chapter III Definition 3.1). Note that $\tilde{G}$ is naturally a $T$-torsor over $E$. There is an induced morphism $\tilde{\pi}\colon \tilde{P}\to E$ which is locally of finite type, and $\tilde{P}$ contains $\tilde{G}.$ We let $\tilde{\mathscr{L}}_{\tilde{P}}:=\tilde{\mathscr{L}}\times^T\mathscr{N}.$ One now checks that the action of $\Gamma$ on $\tilde{G}_\eta$ extends to $\tilde{P}$ and $\tilde{\mathscr{L}}_{\tilde{P}}$, for which we use the definition of the action of $Y$ on $X^\vee\oplus\Z$ (see \cite{Kü}, proof of Lemma 3.7). Intuitively, we want to form the quotient $\tilde{P}/Y$ (where $Y$ acts via $Y\to \Gamma$) to obtain $P$. This is possible in the world of formal schemes: For each $n$, the quotient $(\tilde{P}\times_{\Og_K}\Spec\Og_K/\mathfrak{m}^{n+1})/Y$ exists, is of finite type over $\Spec\Og_K/\mathfrak{m}^{n+1}$, and carries a natural ample line bundle. These schemes define a formal scheme over $\mathrm{Spf}\,\Og_K$, which algebraizes uniquely. This algebraization turns out to be a model of $A$ with all the desired properties. Note in particular that we have an action of $H$ on $P$. \\

\noindent$\mathbf{Remark.}$ From now on, until the last Chapter, we shall assume that the finite group $H$ which acts on $G$ is equal to $\{\mathrm{Id}, [-1]\}$ and that $H$ acts trivially on $\Og_K.$ 
\subsection{Kulikov models}
In this subsection, let $X$ be a smooth, projective, and geometrically integral algebraic surface over $K$ such that $\omega_{X/K}\cong \Og_X.$ This is the case if and only if $X$ is an Abelian surface or a K3 surface.
\begin{definition} Let $X$ be a geometrically integral smooth projective algebraic surface over $K$ with trivial canonical bundle. A \rm Kulikov model \it of $X$ is a regular algebraic space $\mathscr{X}$ which is proper and flat over $\Og_K$ with the following properties: \\
(i) The algebraic space $\mathscr{X}$ is a model of $X$.\\ 
(ii) The reduced special fibre $(\mathscr{X}_k)_{\mathrm{red}}$ of $\mathscr{X}$ is a divisor with normal crossings on $\mathscr{X}.$\\
(iii) We have $$\omega_{\mathscr{X}/\Og_K}((\mathscr{X}_k)_{\mathrm{red}})\cong \Og_{\mathscr{X}}.$$\label{Kulikovdefinition}\\
We say that $\mathscr{X}$ is a \rm strict Kulikov model \it (called a \rm minimal model \it in \cite{CL}) of $X$ if in addition to $(i), (ii), (iii)$ above, the following conditions are satisfied:\\
(i') The special fibre of $\mathscr{X}$ is a scheme, \\
(ii') The special fibre of $\mathscr{X}$ is reduced and its irreducible components are smooth over $k.$\\
In this case, $\omega_{\mathscr{X}/\Og_K}$ is trivial. 
\end{definition}
The possible special fibres of strict Kulikov models of $K3$ surfaces can be classified as follows:
\begin{proposition}
Let $\mathscr{X}\to \Spec\Og_K$ be a strict Kulikov model of the K3-surface $X$ over $K$. Then the special fibre $\mathscr{X}_k$ (which is a scheme by assumption) belongs to one of the following types:\\
Type I: The scheme $\mathscr{X}_k$ is a smooth K3-surface over $k$.\\
Type II: We have $$\mathscr{X}_k=Y_1\cup...\cup Y_N$$ (where $N\in \N$), such that $Y_1$ and $Y_N$ are rational surfaces and $Y_2,..., Y_{N-1}$ are elliptic ruled surfaces, and where all double curves are rulings. The dual complex of $\mathscr{X}_k$ is a chain with endpoints $Y_1$ and $Y_N$\\
Type III: We have $$\mathscr{X}_k=Y_1\cup...\cup Y_N,$$ such that all $Y_j$ are rational surfaces whose intersections form a chain of rational curves, and such that the dual complex of $\mathscr{X}_k$ is a triangulation of the 2-sphere. 
\end{proposition}
\begin{proof}
See \cite{CL}, Corollary 6.3 and Definition 5.4. 
\end{proof}\\
Let $\rho\colon \Gal(\overline{K}/K)\to \Aut_{\Q_{\ell}} (V)$ be a finite-dimensional continuous $\Q_\ell$-adic representation of $\Gal(\overline{K}/K).$ Suppose that $\rho$ is \it unipotent, \rm i.e., that the operator on $V$ induced by any $\sigma\in \Gal(\overline{K}/K)$ has characteristic polynomial $(x-1)^{\dim V}.$ If $P\subseteq \Gal(\overline{K}/K)$ denotes the wild inertia subgroup, it follows that the Galois representation on $V$ factors through $\Gal(\overline{K}/K)/P$ since the operator induced by any $g\in P$ must both have finite order and be unipotent, so it must be trivial. In particular, the monodromy group $\im \rho$ is pro-cyclic. As usual, we define the \it monodromy operator \rm on $V$ to be
$$N:=\log \sigma=\sum_{m=1}^\infty \frac{(-1)^{m+1}}{m}(\sigma-1)^m,$$ where $\sigma$ is a topological generator of $\im \rho.$ Since $\sigma-1$ is nilpotent, so is $N$. In general, if $N$ is a nilpotent operator on a vector space $V,$ we define the \it nilpotency index \rm of $N$ to be the natural number $m$ such that $N^m=0$ but $N^{m-1}\not=0.$ We follow the convention which puts $N^0=\mathrm{Id}.$ 
\begin{proposition}
Let $X$ be a K3 surface over $K$ and assume that $X$ admits a strict Kulikov model $\mathscr{X}\to \Spec\Og_K$. Then the special fibre of $\mathscr{X}_k$ is of type I (resp. type II, type III) if and only if the nilpotency index of the monodromy operator $N_X$ on $H^2_{\et}(X_{\overline{K}}, \Q_\ell)$ is equal to 1 (resp. 2,3). 
\end{proposition}
\begin{proof}
This is \cite{CL}, Theorem 6.4.
\end{proof}\\
There is an analogous classification of the possible special fibres of strict Kulikov models of Abelian surfaces, and also an analogue of the previous Proposition for Abelian surfaces; see \cite{CL}, Definition 5.6, Corollary 8.2, and Theorem 8.3 for more details. In this paper, the following criterion will play an important role:
\begin{proposition}
Let $X$ be a geometrically integral smooth projective algebraic surface over $K$ with trivial canonical bundle. Let $\mathscr{X}\to \Spec\Og_K$ be a flat projective model of $X,$ \label{codim2proposition} where $\mathscr{X}$ is a regular scheme. Suppose further that $\mathscr{X}$ has strictly semistable reduction, and that there is an open subscheme $U$ of $P$  which is smooth over $\Og_K,$ whose complement has codimension $\geq 2$ in $P,$ and which admits a no-where vanishing global 2-form. Then $\mathscr{X}$ is a strict Kulikov model of $X$.
\end{proposition}
\begin{proof}
First note that properties (i) and (ii), as well as (i') and (ii'), of Definition \ref{Kulikovdefinition} are satisfied by assumption, so all we need to show is that $\omega_{\mathscr{X}/\Og_K}\cong \Og_\mathscr{X}.$ It follows also from our assumptions that the morphism $\mathscr{X}\to \Spec \Og_K$ is an l.c.i. morphism, which implies that $\omega_{\mathscr{X}/\Og_K}$ is indeed a line bundle, rather than a complex. The restriction of $\omega_{\mathscr{X}/\Og_K}$ to $U$ is isomorphic to $\bigwedge^2\Omega^1_{U/\Og_K}$ because $U$ is smooth over $\Og_K,$ and by assumption there is a no-where vanishing global section of $\bigwedge^2\Omega^1_{U/\Og_K}$, so $\bigwedge^2\Omega^1_{U/\Og_K}\cong \Og_U.$ Now let $j\colon U\to \mathscr{X}$ denote the open immersion. Because the complement of $U$ in $\mathscr{X}$ has codimension $\geq 2$, it follows that
$$\omega_{\mathscr{X}/\Og_K}\cong j_\ast \omega_{U/\Og_K}\cong j_\ast\Og_U\cong \Og_{\mathscr{X}}$$ using \cite{Ha}, Proposition 1.6 together with the observation that line bundles are always reflexive sheaves.
\end{proof}\\
This criterion can, for example, be used to deduce
\begin{corollary}\rm (Compare \cite{HN3}, Theorem 5.1.6) \it \label{Pkulikovcorollary}
Let $A$ be the Abelian surface we introduced at the beginning, and let $(G, \mathscr{L}, \mathscr{M})\in \mathrm{DEG}^{\mathrm{split}}_{\mathrm{ample}}$ with $G=\mathscr{A}^0$. Suppose further that $(X,Y,\phi, a,b)=\mathrm{For}(F((G,\mathscr{L}, \mathscr{M})))$ and that we have a smooth $\Gamma$-admissible rational polyhedral cone decomposition of $\mathscr{C}\subseteq X^\vee_{\mathbf{R}}\oplus\mathbf{R}$ which satisfies the conditions (a),..., (d) from Theorem \ref{Pexistencetheorem}. Then the model $P$ of $A$ constructed in that Theorem is a strict Kulikov model of $A$. In particular, Abelian surfaces potentially admit strict Kulikov models which are schemes.
\end{corollary}
\begin{proof}
Because the special fibre of $P$ is a reduced divisor with strict normal crossings, we know that the complement of the open subscheme $P^{\mathrm{sm}}$ in $P$ has codimension $\geq 2$. By Theorem \ref{Pexistencetheorem} (i), we know that $P^{\mathrm{sm}}$ is isomorphic to the Néron model $\mathscr{A}$ of $A$. By \cite{BLR}, Chapter 4.2, Corollary 3, we know that there exists a global no-where vanishing 2-form on $\mathscr{A}.$ Hence the claim follows from Proposition \ref{codim2proposition}. The second claim follows because after a finite extension, a rational polyhedral cone decomposition with the properties required for the first part of this Corollary can always be constructed by Proposition \ref{semconeexistenceproposition}.
\end{proof}\\
Suppose $A$ is an (arbitrary) Abelian variety over $K$. Let $\mathscr{A}\to \Spec \Og_K$ denote its Néron model, and let $\mathscr{A}^0$ be the identity component of the Néron model. Then there exist nonnegative integers $r_1, r_2$ and an exact sequence 
$$0\to U\times_k \mathbf{G}_{\mathrm{m}}^{r_2}\to \mathscr{A}^0_k\to B\to 0,$$ where $U$ us a unipotent algebraic group of dimension $r_1,$ and $B$ is an Abelian variety over $k.$ At this point we use that $k$ is algebraically closed, and hence perfect. We say that $A$ has \it semiabelian reduction \rm if $r_1=0,$ and we call $r_2$ the \it toric rank \rm of $A,$ which will sometimes also be denoted by $t$ or $t(A).$ 
\section{Models of Kummer surfaces}
As before, let $A$ be an Abelian surface over $K$, and let $X$ be the Kummer surface associated with $A$. In this section, we shall prove that, after replacing $K$ by one of its finite extensions if necessary, $X$ admits a strict Kulikov model $\mathscr{X}\to \Og_K$. For this purpose we may assume without loss of generality that $A$ has semiabelian reduction over $K$ and that the $K$-group scheme $A[2]$ is constant. Let $(G, \mathscr{L}, \mathscr{M})\in \mathrm{DEG}^{\mathrm{split}}_{\mathrm{ample}}$ with $G=\mathscr{A}^0$ and suppose that the finite group $H:=\{\mathrm{Id}, [-1]\}$ acts on $(G, \mathscr{L}, \mathscr{M})$ in such a way that the action of $[-1]$ on $G$ is multiplication by $-1.$
Letting $(X,Y,\phi, a,b)=\mathrm{For}(F((G, \mathscr{L}, \mathscr{M}))),$ we obtain an action of $H$ on $(X,Y,\phi,a,b)$. Let $\Gamma:=Y\rtimes H.$ By Proposition \ref{semconeexistenceproposition}, we may also assume without loss of generality that there exists a smooth $\Gamma$-admissible rational polyhedral cone decomposition of $\mathscr{C}\subseteq X^\vee_{\mathbf{R}}\oplus \mathbf{R}$ which has properties (a),...,(d) from Theorem \ref{Pexistencetheorem}. Then that Theorem provides us with a regular model $P$ of $A$ over $\Og_K$ which is projective and flat over $\Og_K$ and which has the property that the action of $H$ on $A$ extends uniquely to $P$. We shall now study the fixed locus of this action. We already know that $\mathscr{A}\subseteq P$, and because $A[2]$ is constant, it extends to the closed subscheme $\mathscr{A}[2]$ of $P$. The main technical result of this section will be the observation that the fixed locus of the action of $H$ on $P$ is, in fact, equal to $\mathscr{A}[2].$ This will allow us to perform the Kummer construction on $P$, and we shall see that the resulting scheme $\mathscr{X}$ is a strict Kulikov model of $X$. As indicated introduction, the functor $F$ associates to $(G, \mathscr{L}, \mathscr{M})$ a tuple
$(E, X, Y, \phi, c, c^t, \tilde{G}, \iota, \tau, \tilde{\mathscr{L}}, \mathscr{M}, \lambda_E, \psi, a, b)\in \mathrm{DD}^{\mathrm{split}}_{\mathrm{ample}},$ on which the finite group $H$ acts. We shall have to look at some of these objects more closely; all the details can be found in \cite{Kü}, (2.2)-(2.8). In this tuple, $E$ is an Abelian scheme (of relative dimension 0,1, or 2) over $\Og_K.$ Let $\mathscr{P}_E$ be the rigidified Poincaré-bundle on $E\times_{\Og_K} E^\vee$, where $E^\vee$ denotes the dual Abelian variety. Then $c$ and $c^t$ denote homomorphisms
$c\colon X\to E^\vee(\Og_K)$ and $c^t\colon Y\to E(\Og_K)$ which encode the Raynaud extension (\ref{Raynaudext}) and its dual. Note that $\mathscr{P}_E$ comes with a natural structure of a $\Gm$-biextension of $E\times_{\Og_K} E^\vee$. We can view both $X\times Y$ and $Y$ as group schemes over $K$, which we shall denote by $(Y\times X)_\eta$ and $Y_\eta$, respectively. Then
$$\tau\colon \boldsymbol{1}_{(Y\times X)_\eta}\overset{\cong}{\to}(c\times c^t)^\ast \mathscr{P}_{E, \eta}^{-1}$$ is a trivialization (where $\boldsymbol{1}_{(Y\times X)_\eta}$ stands for the trivial $\Gm$-biextension of $(Y\times X)_\eta).$ Such a trivialization determines (and is determined by) an embedding
$\iota\colon Y\to \tilde{G}(K)$ such that $\pi\circ \iota=c^t$ (see \cite{Kü}, p. 173, (10)). Furthermore, $$\psi\colon \boldsymbol{1}_{Y_\eta}\overset{\cong}{\to} \iota^\ast \tilde{\mathscr{L}}^{-1}_{\eta}$$ is a trivialization of cubical line bundles, where $\boldsymbol{1}_{Y_\eta}$ denotes the trivial cubical line bundle on $Y_\eta.$ For each $(y,\xi) \in Y\times X$, the trivialization $\tau$ identifies $(c(y), c^t(\xi))^\ast\mathscr{P}_{E,\eta}^{-1}$ with a fractional ideal of $\Og_K$, which is equal to $\mathfrak{m}_K^{b(y,\xi)}$ by the definition of $b$. The function $a\colon Y\to \Z$ is constructed similarly using the trivialization $\psi.$ As a fist step towards understanding the fixed locus of the action of $H$ on $P$, we have the following
\begin{lemma} \label{squarelemma}
Let $A$ be an Abelian surface over $K$ with semiabelian reduction and such that $A[2]$ is constant over $K.$ Let $(G, \mathscr{L}, \mathscr{M})\in \mathrm{DEG}^{\mathrm{split}}_{\mathrm{ample}}$ with $G=\mathscr{A}^0$ and such that the finite group $H$ acts on this object as before. Then the embedding $Y\to \tilde{G}(K)$ coming from the object $F((G, \mathscr{L}, \mathscr{M}))\in\mathrm{DD}^{\mathrm{split}}_{\mathrm{ample}}$ has the following property: For each $y\in Y,$ there exists an $x\in \tilde{G}(K)$ such that $\iota(y)=x^2.$
\end{lemma}
\begin{proof}
It follows from \cite{FC}, Chapter III, Theorem 5.9 that $A[2]=G_{\eta}[2]$ can be described as follows:
For each $y\in Y,$ let $Z_{y}$ be pre-image of the point $\iota(y)\in \tilde{G}(K)$ under the map $\tilde{G}\to \tilde{G}$ given by multiplication by 2. The schemes $Z_{y}$ and $Z_{y+2z}$ are canonically isomorphic for any $z\in Y$, and we obtain an isomorphism of schemes
$$G_{\eta}[2]\cong \coprod_{[y]\in Y/2Y} Z_{y}.$$ For $G_{\eta}[2]$ to be a constant $K$-group scheme, it is therefore necessary that for each $y\in Y$, there exist $x\in \tilde{G}(K)$ such that $\iota(y)=x^2$.
\end{proof}
\subsection{Evaluating points of $\tilde{G}$ at characters of $T$}
We have already seen that the group $H$ preserves the stratification of the special fibre of the model $P$ from Theorem \ref{Pexistencetheorem}, and that the action on the set of strata is given in terms of the pairing $b\colon Y\times X \to \Z.$ What will enable us to deduce that $\mathscr{A}[2]\subseteq P$ is already the full fixed locus of the action of $H$ on $P$ is the observation that for all $(y, \xi)\in Y\times X,$ the integer $b(y, \xi)$ is even. In order to deduce this from the previous Lemma, we will need a way of \it evaluating points of $\tilde{G}(K)$ at characters of $T$. \rm More precisely, we shall construct, for each $\xi\in X=X^\ast(T),$ a homomorphism 
$$\ev_\xi\colon \tilde{G}(K)\to \Z,$$ which has the property that for all $(y, \xi)\in Y\times X,$ we have
$$b(y, -\xi)=\ev_{\xi}(\iota(y)).$$ This, together with Lemma \ref{squarelemma}, will imply the claim.  Let $x\in \tilde{G}(K).$ Then the image $j_\eta$ of $x$ under the morphism $\pi\colon \tilde{G}\to E$ extends uniquely to a section $j$ of $E$ over $\Og_K$. Now let $\xi\in X$, and consider the morphism of extensions
\begin{align}\begin{CD} \label{CD}
0@>>>T@>>>\tilde{G}@>>>E@>>>0\\
&&@V{\xi}VV@V{\tilde{\xi}}VV@VV{\mathrm{Id}}V\\
0@>>>\Gm @>>> \Og_{-\xi} @>>> E@>>>0.
\end{CD}\end{align}
We shall continue passing freely between line bundles and their associated $\Gm$-torsors as usual. However, the distinction will play more of a role in this chapter, so the reader is advised always to keep in mind which of the two objects is being referred to. We shall try and keep the notation as clear as possible so that no confusion can arise. 
The fact that the diagram above is a homomorphism of extensions follows from \cite{FC}, p. 43. The point $x$ gives rise to a point in the fibre of $\Og_{-\xi}$ above $j_\eta$, which we shall also denote by $x$. This point, in turn, gives rise to an isomorphism of line bundles
$$j_\eta^\ast \Og_{-\xi, \eta}\to\Og_\eta$$ on $\Spec K,$ defined by $x\mapsto 1.$ This isomorphism identifies the $\Og_K$-module $\Gamma(\Spec \Og_K, j^\ast \Og_{-\xi})$ with a fractional ideal of $\Og_K$, which is equal to $\mathfrak{m}_K^{\ev_{\xi}(x)}$ for some $\ev_{\xi}(x)\in \Z.$ 
\begin{lemma}
Let $\xi\in X$. Then the map 
$$\ev_\xi\colon \tilde{G}(K)\to \Z$$ is a homomorphism. \label{homomorphismlemma}
\end{lemma}
\begin{proof}
Let $x_1, x_2\in \tilde{G}(K).$ Note that the scheme underlying the $\Gm$-torsor $\Og_{-\xi}$ comes with a natural structure of a group scheme over $\Og_K$ such that the map $\tilde{\xi}$ is a homomorphism of group schemes. This follows from the fact that the $\Gm$-torsor $\Og_{-\xi}$ fits into the diagram (\ref{CD}). Denote the images of $x_i$ in $\Og_{-\xi}$ also by $x_i$. For $i=1,2$, let $j_{i, \eta}$ be the image of $x_i$ under $\pi\colon \tilde{G}\to E$. The morphisms $j_{i,\eta}$ extend uniquely to $\Og_K$-sections $j_i$ of $E$. For each $i=1,2$ choose a no-where vanishing section $\epsilon_i$ of the line bundle $j_i^\ast \Og_{-\xi}$ on $\Spec\Og_K$. Also recall that the line bundle $\Og_{-\xi}$ comes with a natural rigidification, that is a no-where vanishing section $1_{-\xi}$ of $e^\ast\Og_{-\xi}$, where $e\colon \Spec\Og_K\to E$ denotes the identity section.
For any global section $j\in E(\Og_K)$, denote by $T_j\colon E\to E$ the morphism given by translation by $j$. Let us first prove that there is an isomorphism 
\begin{align}\Phi\colon T_{j_1+j_2}^\ast\Og_{-\xi}\otimes \Og_{-\xi}\to T_{j_1}^\ast\Og_{-\xi}\otimes T_{j_2}^\ast\Og_{-\xi} \label{iso} \end{align}  sending $T_{j_1+j_2}^\ast(\epsilon_1\epsilon_2)\otimes 1_\xi$ to $T_{j_1}^\ast(\epsilon_1)\otimes T_{j_2}^\ast(\epsilon_2).$ The existence of an isomorphism as in (\ref{iso}) (not necessarily satisfying the second condition) follows from the theorem of the cube (\cite{FC}, Chapter I, Theorem 1.3). Because the sections $\epsilon_i$ were chosen to be no-where vanishing, the elements $T_{j_1+j_2}^\ast(\epsilon_1\epsilon_2)\otimes 1_{-\xi}$ and $T_{j_1}^\ast(\epsilon_1)\otimes T_{j_2}^\ast(\epsilon_2)$ of $\Gamma(\Spec \Og_K, (j_1+j_2)^\ast \Og_{-\xi}\otimes e^\ast\Og_{-\xi})$ and $\Gamma(\Spec\Og_K, j_1^\ast\Og_{-\xi}\otimes j_2^\ast \Og_{-\xi})$ are generators of their respective modules. This implies that any such isomorphism will satisfy the second condition after scaling by a unit of $\Og_K.$ Now choose, for $i=1,2$, elements $\lambda_i\in K^\times$ such that $x_i=\lambda_i\epsilon_{i, \eta}$. We shall use the fact that the $\lambda_i$ can be seen both as elements of $K^\times$ and as $K$-points of $\Og_{-\xi}$ via the embedding $\Gm\to \Og_{-\xi}$ from the diagram (\ref{CD}), and that multiplying the element $\epsilon_{i, \eta}\in \Og_{-\xi}(K)$ by $\lambda_i$ using the $\Og_K$-group structure of $\Og_{-\xi}$ and multiplying the global section $\epsilon_{i, \eta}$ of $j_{i, \eta}^\ast \Og_{-\xi, \eta}$ by $\lambda_i$ using the $K$-vector space structure of $\Gamma(\Spec K, j_{i, \eta}^\ast \Og_{-\xi, \eta})$ has the same effect. Indeed, using this last observation, we find that under the isomorphism $\Phi_\eta,$ 
\begin{align*}
T_{j_{1,\eta}+j_{2, \eta}}^\ast (x_1x_2)\otimes 1_{-\xi, \eta}&=\lambda_1\lambda_2 T_{j_{1,\eta}+j_{2, \eta}}^\ast (\epsilon_{1, \eta}\epsilon_{2, \eta})\otimes 1_{-\xi, \eta}\\
&\mapsto \lambda_1\lambda_2 T_{j_{1, \eta}}^\ast(\epsilon_{1, \eta})\otimes T_{j_{2, \eta}}^\ast(\epsilon_{2, \eta})\\
&=T_{j_{1, \eta}}^\ast(x_1)\otimes T_{j_{2, \eta}}^\ast(x_2).
\end{align*}
This implies in particular that the diagram
$$\begin{CD}
(j_{1,\eta}+j_{2, \eta})^\ast \Og_{-\xi, \eta} \otimes e_{\eta}^\ast\Og_{-\xi, \eta} @>{e_\eta^\ast\Phi_\eta}>>j_{1, \eta}^\ast\Og_{-\xi, \eta}\otimes j_{2, \eta}^\ast\Og_{-\xi, \eta}\\
@V{x_1x_2}VV\!\!\!\!\!\!\!\!\!\!\!\!\!\!\!\!\!\!\!\!\!\!\!\!\!\!\!\!\!\!\!\!\!\!\!\!\!\!\!\!\!\!\!\!\!\!\!\!\!\!\!\!\!\!\!\!\!\!\!\!\!\!\!\!\!\!\!\!\!\!\!\!\!\!\!\!\!\!\!\!\!\!\!\!\!\!\!\!\!\!\!\!\!\!\!\!\!\!\!\!\!\!\!\!\!\!\!\!\!\!\!\!\!\!\!\!@V{1_{-\xi, \eta}}VV \!\!\!\!\!\!\!\!\!\!\!\!\!\!\!\!\!\!\!\!\!\!\!\!\!\!\!\!\!\!\!\!\!\!\!\!\!\!@V{x_1}VV\hspace{0.4 in}@VV{x_2}V\\
{}\hspace{0.7 in}\Og_\eta\hspace{0.2 in}\otimes \hspace{0.2 in}\Og_\eta @>>>  \hspace{0.12 in} \Og_\eta\hspace{0.07 in}\otimes\hspace{0.07 in} \Og_\eta\\
\end{CD}$$
commutes. Since the isomorphism $e_\eta^\ast \Phi_\eta$ extends to the pullback of the isomorphism (\ref{iso}) along $e\colon \Spec\Og_K\to E,$ and the section $1_{-\xi, \eta}$ extends to $1_{-\xi}$, we find that the morphism given by $x_1x_2$ identifies $\Gamma(\Spec\Og_K, (j_1+j_2)^\ast\Og_{-\xi})$ with the product of the fractional ideals of $\Og_K$ given by 
$$\Gamma(\Spec\Og_K, j_i^\ast\Og_{-\xi})\to \Gamma(\Spec K, j_{i, \eta}^\ast\Og_{-\xi, \eta})\xrightarrow{x_i\mapsto 1} K$$ for $i=1,2$. In particular, $\ev_\xi(x_1x_2)=\ev_{\xi}(x_1)+\ev_{\xi}(x_2),$ so the claim follows. 
\end{proof}\\
Now we must prove that for all $\xi\in X$ and $y\in Y$ the equality
$$b(y,-\xi)=\ev_\xi(\iota(y))$$ holds. We shall need the following
\begin{lemma}
Let $T_1,$ $T_2$ be algebraic tori, and let $E$ be an Abelian variety over $K$. Let $G_1$, $G_2$ be commutative algebraic groups over $K$ such that we have a commutative diagram
\begin{align}\begin{CD}
0@>>>T_1@>{i_1}>>G_1@>{\pi_1}>>E@>>>0\\
&&@V{\xi}VV@VV{\tilde{\xi}}V@VV{\mathrm{Id}}V\\
0@>>>T_2@>>{i_2}>G_2@>>{\pi_2}>E@>>>0 \label{diagram}
\end{CD}\end{align}
with exact rows. Suppose that $\gamma\colon G_1\to G_2$ is a homomorphism of torsors on $E$ with respect to $\xi$ (i.e., that $\gamma$ is a morphism of schemes compatible with the actions of $T_1$ and $T_2$ on $G_1$ and $G_2$ in the sense that $\gamma(t\cdot g)=\xi(t)\cdot \gamma(g)$ for $t\in T_1$ and $g\in G_1$) which fits into diagram (\ref{diagram}) instead of $\tilde{\xi}.$ Then $\gamma=\tilde{\xi}.$
\end{lemma}
\begin{proof}
We shall use multiplicative notation for the $T_j$ and $G_j$, and additive notation for $E.$ First notice that $\gamma\tilde{\xi}^{-1}$ factors through $i_2$ since $\pi_2\circ(\gamma\tilde{\xi}^{-1})=\pi_2\circ \gamma-\pi_2\circ \tilde{\xi}=0,$ where we used the commutativity of diagram (\ref{diagram}) and the final assumption for the last equality. Also observe that $\gamma\tilde{\xi}^{-1}$ is constant on the fibres of $\pi_1$ since $\gamma\tilde{\xi}^{-1}(tg)=i_2(\xi(t))\gamma(g)i_2(\xi(t))^{-1}\tilde{\xi}(g)^{-1}=\gamma\tilde{\xi}^{-1}(g)$ for $t\in T_1$ and $g\in G_1.$ In particular, we find that there must exist a map of schemes $\lambda\colon E\to T_2$ such that $\gamma\tilde{\xi}^{-1}=i_2\circ \lambda\circ \pi_1.$ To conclude, all we have to show is that $\lambda$ is constant with value 1. Since $E$ is connected and proper over $K$ and $T_2$ is affine, $\lambda$ must be constant. Suppose $e$ is the neutral element of $E$. Then the value $i_2(\lambda(e))$ is equal to $\gamma\tilde{\xi}^{-1}(i_1(t)),$ for any $t\in T_1.$ We find $\gamma\tilde{\xi}^{-1}(i_1(t))=\gamma(i_1(t))\tilde{\xi}(i_1(t))^{-1}=i_2(\xi(t))i_2(\xi(t))^{-1}=1.$ Hence the claim follows. 
\end{proof}\\
We are now ready to prove
\begin{proposition}
For all $\xi\in X$ and $y\in Y$ we have \label{bevproposition}
$$b(y,-\xi)=\ev_{\xi}(\iota(y)).$$
\end{proposition}
\begin{proof}
Observe that, for all $\xi\in X,$ there is a perfect pairing $\Og_\xi\otimes \Og_{-\xi}\to \Og_E.$ We shall first show that the image of a local section $\chi$ of $\pi\colon\tilde{G}\to E$ under $\tilde{\xi}\colon \tilde{G}\to \Og_{-\xi}$ acts on a local section $f$ of $\Og_{\xi}$ as $f\mapsto \chi^\ast f.$ By the Lemma preceding this Proposition, all we have to show is that the map $\tilde{G}\to \Og_{-\xi}$ given by $\chi\mapsto (f\mapsto \chi^\ast f)$ fits into diagram (\ref{CD}) instead of $\tilde{\xi}.$ This will follow if we can show that it induces the map $\xi\colon T\to \Gm$ on the fibre above the identity of $E$. But this is the case, because the image of $1$ under the closed immersion $\Gm\to \Og_{-\xi}$ acts on a local section $f$ of $\Og_\xi$ (on an open set of $E$ which contains the identity $e\in E$) as $f\mapsto f(1),$ and we have $e^\ast\chi^\ast f=f(\chi(e))=\xi(\chi(e))f(1).$\\
Now suppose we have a point $x\in \tilde{G}(K).$ Denote the image of $x$ under $\tilde{G}(K)\to E(K)$ by $j_\eta.$ The point $x$ determines (and is determined by) a homomorphism of quasi-coherent $\Og_{E_{\eta}}$-algebras
$$\pi_\ast\Og_{\tilde{G}_{\eta}}=\bigoplus_{\xi\in X} \Og_{\xi, \eta} \to j_{\eta\ast}\Og_\eta,$$ which, in turn, is given by a compatible system of non-zero elements $\delta_\xi\in j_\eta^\ast\Og_{\xi, \eta}^\vee.$ As a next step, we shall show that the image of $x$ under the map $\tilde{\xi}\colon \tilde{G}\to \Og_{-\xi}$ from diagram (\ref{CD}) is precisely $\delta_{\xi}.$ To see this, suppose we have a local section $\chi$ of $\pi\colon \tilde{G}_\eta\to E_\eta$ such that $\chi(j_\eta)=x.$ Suppose also that we have a $\xi$-eigenfunction $f$ above the open set on which $\chi$ is defined. We have
\begin{align*}
(j_\eta^\ast \tilde{\xi}(\chi))(j_\eta^\ast f)&=j_\eta^\ast(\tilde{\xi}(\chi)(f))\\
&=j_\eta^\ast\chi^\ast f\\
&=x^\ast f\\
&=\delta_{\xi}(j_\eta^\ast f).
\end{align*}
Hence $j_\eta^\ast \tilde{\xi}(\chi)=\delta_\xi,$ as desired.
We can now proceed to proving our claim. Let $y\in Y.$ By \cite{Kü}, p. 173 (10), the point $\iota(y)$ is given as follows: The trivialization
$$\tau\colon \boldsymbol{1}_{(Y\times X)_{\eta}} \to (c^t \times c)^\ast \mathscr{P}_{E_{\eta}}^{-1}$$ is given by a compatible system of sections $\tau(y, \xi)\in \Gamma(\eta, (c^t(y), c(\xi))^\ast \mathscr{P}_{E, \eta}^{-1}),$ which, in turn, defines a morphism of quasicoherent $\Og_{E_{\eta}}$-algebras
$$\pi_{\ast} \Og_{\tilde{G}_{\eta}}=\bigoplus_{\xi\in X}\Og_{\xi, \eta}=\bigoplus_{\xi} (\mathrm{Id}\times c(\xi))^\ast \mathscr{P}_{E, \eta} \xrightarrow{\tau(y, \xi)} c^t(y)_\ast \Og_{\eta}.$$ In particular, we see that, if $x=\iota(y),$ we have
$$\delta_{\xi}=\tau(y, \xi)$$ for all $\xi\in X.$  Now let $\xi\in X.$ By the definition of the pairing $b(-,-),$ (\cite{Kü}, p.174, 2.3), the number $b(y, -\xi)$ is the unique integer such that the fractional ideal of $\Og_K$ given by the image of
\begin{align}\Gamma(\Spec \Og_K, (c^t(y), c(-\xi))^\ast \mathscr{P}_{E})\subseteq \Gamma(\eta, (c^t(y), c(-\xi))^\ast \mathscr{P}_{E, \eta})\xrightarrow{\cong} K,\label{fracidealformula}\end{align} where the last isomorphism is $\tau(y, -\xi)=\delta_{-\xi},$ is equal to $\mathfrak{m}_K^{b(y, -\xi)}.$ However, we have a canonical isomorphism $(c^t(y), c(-\xi))^\ast \mathscr{P}_{E}=c^t(y)^\ast(\mathrm{Id}\times c(-\xi))^\ast \mathscr{P}_{E}=c^t(y)^\ast \Og_{-\xi}.$ By what we have done in the first part of this proof, we know that the isomorphism in (\ref{fracidealformula}) is the same as the isomorphism $\Gamma(\eta, c^t(y)^\ast \Og_{-\xi, \eta})\to K$ given by $\iota(y)\mapsto 1$ (it follows from the compatibility between the various $\delta_{\xi}$ that, modulo the canonical isomorphism $j_{\eta}^\ast \Og_{-\xi}=j_{\eta}^\ast \Og_{\xi}^\vee,$ we have $\delta_{-\xi}(\delta_{\xi})=1$). By definition, $\ev_{\xi}(\iota(y))$ is the unique integer such that the fractional of $\Og_K$ given by the image of 
$$\Gamma(\Spec \Og_K, c^t(y)^\ast \Og_{-\xi})\subseteq \Gamma(\eta, c^t(y)^\ast \Og_{-\xi, \eta})\xrightarrow{\cong} K,$$ where the last isomorphism is given by $\iota(y),$ is equal to $\mathfrak{m}_K^{\ev_{\xi}(y)}.$ Hence we find $\mathfrak{m}_K^{b(y, -\xi)}=\mathfrak{m}_K^{\ev_{\xi}(\iota(y))},$ so the claim follows. 
\end{proof}\\
We have now arrived at
\begin{proposition}
The assumption that $A[2]$ be constant implies that, for all $\xi\in X$ and $y\in Y,$ the integer $b(y,\xi)$ is even. \label{evenproposition}
\end{proposition}
\begin{proof}
Let $\xi\in X$ and $y\in Y.$ By Lemma \ref{squarelemma}, there exists a point $x\in \tilde{G}(K)$ such that $\iota(y)=x^2.$ By Proposition \ref{bevproposition}, we have
$b(y, \xi)=-\ev_\xi(\iota(y))=-2\ev_\xi(x).$
\end{proof}
\subsection{The final step}
We have now assembled all the technical tools which we need to prove our main result. Let $A$ be an Abelian surface over $K$. We may assume without loss of generality that $A$ has semiabelian reduction over $K$ and that $A[2]$ is constant. Choose an object 
$$(G, \mathscr{L}, \mathscr{M})\in \mathrm{DEG}^{\mathrm{split}}_{\mathrm{ample}},$$ where $G=\mathscr{A}^0,$ on which the finite group $H=\{\mathrm{Id}, [-1]\}$ acts. As before, $\mathscr{A}$ denotes the Néron model of $A$. Using the functor
$F\colon \mathrm{DEG}^{\mathrm{split}}_{\mathrm{ample}}\to \mathrm{DD}^{\mathrm{split}}_{\mathrm{ample}}, $ we associate to $(G, \mathscr{L}, \mathscr{M})$ split ample degeneration data. In particular, we obtain the object $(X,Y,\phi,a,b)\in \mathcal{C}$ introduced before. One easily convinces oneself that the induced action of $[-1]$ on $X$ is given by multiplication by $-1.$ After passing to a finite extension of $K$ if necessary, we may assume that there is a smooth rational polyhedral cone decomposition $\{\sigma_\alpha\}_{\alpha\in I}$ of $\mathscr{C}:=(X^\vee_\mathbf{R}\times \mathbf{R}_{>0}) \cup \{0\} \subseteq X^\vee_\mathbf{R}\oplus \mathbf{R}$ which has the properties (a),..., (d) listed in Corollary \ref{Pexistencetheorem}; this follows from Proposition \ref{semconeexistenceproposition}. As in the previous sections, $P$ will denote the model of $A$ from Theorem \ref{Pexistencetheorem}.
What follows now is a central step in our argument: 
\begin{lemma}
Let $\{\sigma_{\alpha}\}_{\alpha\in I}$ be a smooth $\Gamma$-admissible rational polyhedral cone decomposition of $\mathscr{C}.$ Let $\alpha\in I$ such that $\sigma_{\alpha}$ is either two-dimensional or three-dimensional, or that it is one-dimensional with primitive element $(\ell, s)$, where $s>1.$ Then there does not exist an element $y\in Y$ such that \label{conefreelemma}
$$S_{(y, \mathrm{Id})}(\sigma_{\alpha})=S_{(0,[-1])}(\sigma_{\alpha}).$$ In other words, the group $H=\{\mathrm{Id}, [-1]\}$ acts freely on the set of equivalence classes $J/Y$, where $J\subset I$ is the set of indices $\beta$ such that $\sigma_{\beta}$ is two-dimensional or three-dimensional, or one-dimensional with primitive element $(\ell, s)$, $s>1.$
\end{lemma}
\begin{proof}
Let $\alpha\in I$, and assume that $\sigma_{\alpha}=\mathbf{R}_{\geq 0}(\ell,s),$ where $\ell\in X^\vee,$ $s\in \N.$ We may assume that there does not exist any integer different from $\pm 1$ which divides both $\ell$ and $s$. Suppose there is $y\in Y$ such that $S_{(y, \mathrm{Id})}(\sigma_{\alpha})=S_{(0,[-1])}(\sigma_{\alpha}).$ Then $(\ell+sb(y,-),s)=(-\ell,s),$ so $2\ell=-sb(y,-).$ By Proposition \ref{evenproposition}, we can find $f\in X^\vee$ such that $b(y,-)=-2f.$ This implies that $\ell=sf,$ so $s=1.$

 Suppose now that $\sigma_{\alpha}$ is two-dimensional. We can find $\ell_1, \ell_2\in X^\vee$ such that $\sigma_{\alpha}=\mathbf{R}_{\geq 0}(\ell_1,s_1)+\mathbf{R}_{\geq 0}(\ell_2,s_2)$, where the $\ell_j\in X^\vee$ and $s_j\in \N$ have been chosen such that $(\ell_1,s_1), (\ell_2, s_2)$ generate a submodule of $X^\vee\oplus\Z$ of rank 2, and such that the quotient is torsion-free. Now suppose that $S_{(y, \mathrm{Id})}(\sigma_{\alpha})=S_{(0,[-1])}(\sigma_{\alpha})$ for some $y\in Y.$ This happens if and only if either
\begin{align}
\ell_j+s_jb(y,-)=-\ell_j \label{case1}
\end{align}
for $j=1,2,$ or
\begin{align}
\ell_j+s_jb(y,-)=-\ell_{3-j} \label{case2}
\end{align}
and $s_1=s_2$, for $j=1,2.$
For case (\ref{case1}), we may again choose $f\in X^\vee$ such that $b(y,-)=-2f$. We find, for $j=1,2,$ that $\ell_j=s_jf.$ But then the submodule of $X^\vee\oplus\Z$ generated by $(\ell_1, s_1)=s_1(f,1)$ and $(\ell_2, s_2)=s_2(f,1)$ has rank $1$, a contradiction. In case (\ref{case2}), we write $s$ for both $s_1$ and $s_2$ and find
$$sb(y,-)=-\ell_1-\ell_2.$$ Again, we pick $f\in X^\vee$ such that $b(y,-)=-2f.$ Then $\ell_1+\ell_2=2sf,$ so
$$(\ell_1,s)+(\ell_2,s)=2(sf,s).$$ Since the $\Z$-module $X^\vee\oplus\Z/\langle(\ell_1,s), (\ell_2,s)\rangle$ is torsion-free, this implies that there exist $\lambda_1, \lambda_2\in \Z$ such that
$$\lambda_1\ell_1+\lambda_2\ell_2=sf$$ and $\lambda_1+\lambda_2=1.$ This implies that
$$2\lambda_1\ell_1+2(1-\lambda_1)\ell_2=\ell_1+\ell_2,$$ and hence
$$(2\lambda_1-1)\ell_1+(1-2\lambda_1)\ell_2=0.$$ Since $2\lambda_1-1$ can never vanish, this implies that $\ell_1=\ell_2$, so that $(\ell_1,s_1)=(\ell_2,s_2),$ another contradiction.
Finally, assume that $\sigma_{\alpha}$ is three-dimensional, such that there are $\ell_1,\ell_2, \ell_3$ with the property that 
$$\sigma_\alpha=\mathbf{R}_{\geq 0}(\ell_1,s_1)+\mathbf{R}_{\geq 0}(\ell_2,s_2)+\mathbf{R}_{\geq 0}(\ell_3,s_3),$$ and such that $(\ell_1,s_1), (\ell_2,s_2), (\ell_3,s_3)$ form a basis of $X^\vee\oplus\Z.$ In order for $S_{(y, \mathrm{Id})}(\sigma_{\alpha})=S_{(0,[-1])}(\sigma_{\alpha})$ to be satisfied for some $y\in Y$, there must be a permutation $\gamma\in S_3$ such that
$$\ell_j+s_jb(y,-)=-\ell_{\gamma(j)}$$ and $s_j=s_{\gamma(j)}$ for $j=1,2,3.$ If $\gamma$ is the identity, we derive a contradiction as in case (\ref{case1}) above. If $\gamma$ swaps two of the indices, then the third must be kept fixed, and the other two generate a face of $\sigma_{\alpha}.$ In this case, we can derive a contradiction as in case (\ref{case2}) above. If $\gamma$ has no fixed points, we obtain $s_1=s_2=s_3$ and
$$\ell_1+\ell_2=\ell_1+\ell_3=\ell_2+\ell_3,$$ which implies $\ell_1=\ell_2=\ell_3.$ Hence the claim follows in general. 
\end{proof}\\
We are now in a position to prove the following result:
\begin{theorem}
Let $P$ be the regular projective model of $A$ from Theorem \ref{Pexistencetheorem}, and let $[-1]$ denote the action of $P$ which extends the action of $[-1]$ on the open subscheme $G$ of $P.$ Recall that $P$ contains the Néron model $\mathscr{A}$ of $A.$ Then the fixed locus of $[-1]$ on $P$ coincides with the closed subscheme $\mathscr{A}[2]\subseteq P.$ \label{fixedlocustheorem}
\end{theorem}
\begin{proof}
If there were any fixed points of $[-1]$ not contained in $\mathscr{A}[2]$, they would have to lie on the special fibre of $P$. We know that the special fibre of $P$ has a stratification indexed by $I^+$ as described in Theorem \ref{Pexistencetheorem}. By Lemma \ref{conefreelemma}, we know that the fixed points cannot lie on strata associated with two-dimensional or three-dimensional cones. However, we also know from Theorem \ref{Pexistencetheorem} that the strata associated with one-dimensional cones are contained in $P^{\mathrm{sm}},$ which is precisely the image of the open immersion $\mathscr{A}\to P.$ Hence the fixed points are contained in $\mathscr{A}$, and hence in $\mathscr{A}[2].$
\end{proof}\\
We are now in a position to prove the main result of the first part. We need the following lemma about quotients by finite groups:
\begin{lemma}
Let $U$ be a scheme which is flat and quasi-projective over $\Og_K.$ Suppose that the finite group $H$ acts on $U$ (respecting the $\Og_K$-structure). Assume further that $\#H$ is invertible in $\Og_K$. Then taking the quotient of $U$ by $H$ commutes with base change along the morphism $\Spec k\to \Spec \Og_K.$ \label{quotientlemma}
\end{lemma}
\begin{proof}
It is well-known that $U$ can be covered by finitely many $H$-stable affine open subschemes $U_1,..., U_m$, and that the quotient $U/H$ is given by gluing the schemes $U_j/H$. Hence we may assume, without loss of generality, that $U$ is affine, and equal to $\Spec B$, say. In this case, all we have to prove is that the canonical morphism 
$$B^H\otimes_{\Og_K} k \to (B\otimes_{\Og_K} k)^H$$ is an isomorphism. To see that this morphism is injective, suppose $b\in B^H$ such that $b\otimes 1$ is equal to zero in $B\otimes_{\Og_K} k$. This means that $b=\pi_K b'$ for some $b'\in B$. Let $h\in H$. Then $\pi_K(b'-h(b'))=0$, so $b'=h(b')$ since $B$ is flat over $\Og_K$, and has therefore no $\pi_K$-torsion. Since $h$ was chosen arbitrarily, this implies that $b'\in B^H$, so $b\otimes1=0$ in $B^H\otimes_{\Og_K}k$. That proves injectivity. To prove surjectivity, let $b\in B$ such that $b\otimes 1$ is $H$-invariant. For each $h\in H$, we have $b\equiv h(b) \mod \pi_K$, so
$$\sum_{h\in H} h(b)\equiv \#H\cdot b\mod\pi_K.$$ Since $\#H\in \Og_K^\times$, this can be re-written as
$$\Bigg((\#H)^{-1}\sum_{h\in H} h(b)\Bigg)\otimes 1=b\otimes 1$$ in $B\otimes_{\Og_K}k.$ Because $(\#H)^{-1}\sum_{h\in H} h(b)$ is visibly $H$-invariant, the claim follows.
\end{proof}
\begin{lemma}
Let $X$ be a two-dimensional smooth (but not necessarily connected) algebraic variety over a field of characteristic different from 2. Let $\iota$ be an involution of $X$ whose scheme-theoretic fixed locus consists of finitely many geometrically reduced points, and such that if $x\in X$ is a fixed point, the action of $\mathrm{d}\iota_x$ on the Zariski tangent space $T_x X$ is given by multiplication by $-1.$ If $Y$ denotes the fixed locus of $\iota$, then the quotient $(\mathrm{Bl}_YX)/\iota$ is smooth over the ground field. \label{quotientsmoothlemma}
\end{lemma}
\begin{proof}
Since taking geometric quotients and blow-ups commutes with flat base change, and since smoothness is local in the fpqc-topology on the target, we may assume without loss of generality that the ground field is algebraically closed. In this situation, the first part of the proof of \cite{B}, Chapter 10, Theorem 10.6 can be taken \it mutatis mutandis \rm to show that $(\mathrm{Bl}_YX)/\iota$ is regular, and hence smooth over the ground field. In fact, only the case where $X$ is an Abelian variety and $\iota=[-1]$ is considered in \it loc. cit. \rm However, the only facts used about the behaviour of the involution are those listed in the Lemma, so the claim follows.
\end{proof}\\
We need one more Lemma:
\begin{lemma}
Let $P$ be the projective model of $A$ from Theorem \ref{Pexistencetheorem}, and let $P'$ be an open subscheme of $P$ containing $\mathscr{A}[2].$ Then the special fibre of $\mathrm{Bl}_{\mathscr{A}[2]}P'$ is $\mathrm{Bl}_{\mathscr{A}_k[2]}P'_k.$ In other words, the blow-up of $P'$ in $\mathscr{A}[2]$ commutes with base change along the morphism $\Spec k\to \Spec \Og_K.$ \label{basechangelemma}
\end{lemma}
\begin{proof}
We can cover $P'$ by open affine subschemes such that the intersection of $\mathscr{A}[2]$ with any one of our open affine subschemes is given by an ideal generated by a regular sequence of length 2. Pick such an affine open subscheme, equal to $\Spec R,$ say. Suppose $\mathscr{A}[2]\cap \Spec R$ is cut out by the ideal $I\subseteq R.$ Then the $R/I$-module $I^n/I^{n+1}$ is free for all $n\in \N,$ so it is in particular flat as an $\Og_K$-module. By induction, we deduce that, for all $n\in \N,$ $R/I^n$ is flat as an $\Og_K$-module. This, in turn, implies that the induced map $I^n\otimes_{\Og_K}k\to R\otimes_{\Og_K}k$ is injective for all $n$, and that its image is the $n$-th power of the image of $I\otimes_{\Og_K}k\to R\otimes_{\Og_K}k.$ However, this last image is precisely the ideal defining $(\mathscr{A}[2]\cap \Spec R)_k:=(\mathscr{A}[2]\cap \Spec R)\times_{\Og_K}\Spec k \subseteq \Spec R\otimes_{\Og_K}k.$ Hence we find
$$ \mathrm{Bl}_{(\mathscr{A}[2]\cap \Spec R)_k}(\Spec R\times_{\Og_K}\Spec k)=\Proj \bigoplus_{n\geq 0} (I\otimes_{\Og_K}k)^n=\Big( \Proj \bigoplus_{n\geq 0} I^n \Big) \times_{\Og_K} \Spec k,$$
where the right hand side clearly equals the special fibre of $\mathrm{Bl}_{\mathscr{A}[2]\cap \Spec R}\Spec R.$ Our claim now follows from a gluing argument.
\end{proof}
\begin{proposition}
Let $U$ be the blow-up of the Néron model $\mathscr{A}$ of $A$ in the closed subscheme $\mathscr{A}[2].$ Then the action of $[-1]$ extends to $U$, and the quotient
$U/[-1]$ is smooth over $\Og_K$. \label{quotientsmoothlemmaii}
\end{proposition}
\begin{proof}
The scheme $U/[-1]$ is certainly flat and of finite type over $\Og_K$, so it suffices to prove that the generic fibre and the special fibre of $U/[-1]$ are smooth over $K$ and $k$, respectively. This follows if we can show that both fibres of $U/[-1]\to \Spec\Og_K$ are smooth. By Lemma \ref{quotientlemma}, we know that the special fibre of $U/[-1]$ is the quotient of the $k$-scheme $\mathrm{Bl}_{\mathscr{A}_{k}[2]} \mathscr{A}_{k}$ by $[-1]$ (here we use that the special fibre of $U$ is the blow-up $\mathrm{Bl}_{\mathscr{A}_{k}[2]} \mathscr{A}_{k},$ which follows from Lemma \ref{basechangelemma}).  We also know that the generic fibre of $U$ is the quotient of $\mathrm{Bl}_{A[2]} A$ by $[-1],$ since geometric quotients and blow-ups commute with flat base change. It follows from general theory of group schemes that the conditions of Lemma \ref{quotientsmoothlemma} are satisfied for both the generic and the special fibre of $U$. The claim now follows from Lemma \ref{quotientsmoothlemma}.
\end{proof}\\
We can now prove
\begin{theorem} Let $A$ be an Abelian surface over $K$. Let $X$ be the Kummer surface associated with $A$. Then, after passing to a finite extension of $K$ if necessary, there exists a strict Kulikov model $\mathscr{X}\to \Spec\Og_K$ of $X$. Moreover, $\mathscr{X}$ is a scheme. \label{Kummersemistabletheorem}
\end{theorem}
\begin{proof}
Without loss of generality, we may assume that $A$ has semiabelian reduction already over $K$, and that $A[2]$ is constant over $K$. Assume for the moment that the identity component $\mathscr{A}_k^0$ of the special fibre of the Néron model of $A$ is either isomorphic to $\mathbf{G}_{\mathrm{m}, k}^2$ or an extension of an elliptic curve by $\Gm,$ so that our previous results become applicable. We know from Theorem \ref{Pexistencetheorem} and Proposition \ref{semconeexistenceproposition} that, after passing to a finite extension of $K$ if necessary, there exists a projective regular model $P$ of $A$ which contains the Néron model $\mathscr{A}$ of $A$ and such that the action of $H$ on $G=\mathscr{A}^0$ extends to $P$. Define
$$\tilde{\mathscr{X}}:=\mathrm{Bl}_{\mathscr{A}[2]} P.$$ Then $H$ acts on $\tilde{\mathscr{X}}$ and, in fact, $\tilde{\mathscr{X}}$ admits an open covering of $H$-stable open subsets, one given by $U_0:=\mathrm{Bl}_{\mathscr{A}[2]}\mathscr{A},$ and the other one given by $U_1:=P\backslash \mathscr{A}[2].$ Let us define
$$\mathscr{X}:=\tilde{\mathscr{X}}/H.$$ We shall now show that $\mathscr{X}$ has all the desired properties. First of all, it is indeed a model of $X$, because geometric quotients commute with flat base change. We also note that $\tilde{\mathscr{X}}/H$ is a scheme since $\tilde{\mathscr{X}}$ is projective over $\Og_K$. Now observe that $\mathscr{X}$ admits an open cover by the subschemes $U_0/H$ and $U_1/H$. It follows from Lemma \ref{quotientsmoothlemmaii} that $U_0/H$ is smooth over $\Og_K$. On the other hand, the action of $H$ on $U_1$ is free by Theorem \ref{fixedlocustheorem}. In particular, the morphism $U_1\to U_1/H$ is étale, which implies that the scheme $\mathscr{X}$ is regular (since regularity is local in the étale topology and $U_0/H$, $U_1$ are regular) and that the special fibre of $\mathscr{X}$ is a divisor with normal crossings. In the case where $\mathscr{A}_k^0$ is an Abelian variety, $A$ has good reduction. In other words, the Néron model $\mathscr{A}$ of $A$ is projective over $\Og_K.$ Define $P:=\mathscr{A}.$ Then we obtain an action of $H$ on $P$ as above, and the quotient $\mathscr{X}$ of $\tilde{\mathscr{X}}:=\mathrm{Bl}_{\mathscr{A}[2]}P$ by the action induced by $[-1]$ is smooth over $K$, so it is in particular semistable. This last case has already been observed by Ito \cite{Ma}, Lemma 4.2. It follows from \cite{S}, Lemma 3.2, that the morphism $\mathscr{X}\to \Spec\Og_K$ is indeed projective. Now we must prove that the irreducible components of the special fibre of $\mathscr{X}\to \Spec\Og_K$ are smooth over $k$. Note that any singular points of the $\mathscr{X}_k$ must be contained in $U_1/H$ because $U_0/H$ is smooth over $\Og_K$. Also observe that the special fibre $U_{1,k}$ of $U_1$ still has a stratification indexed by $I^+/Y$ which is preserved by the action of $H$. Here, as before, $I$ denotes the index set of the semistable rational polyhedral cone decomposition $\{\sigma_\alpha\}_{\alpha\in I}$ which we used to construct $P$, and $I^+$ denotes the set of indices belonging to non-zero cones. Observe also that the pre-image of a non-regular irreducible component of the special fibre of $U_1/H$ in $U_{1,k}$ must be the union of two irreducible components of $U_{1,k}$ which intersect non-trivially. This can be seen as follows: The pre-image cannot consist of more than two irreducible components because $\#H=2$. However, if the pre-image either consisted of one irreducible component of $U_{1,k}$ or the union of two disjoint irreducible components of $U_{1,k}$, then its quotient by $H$ would be smooth over $k$ (since the action of $H$ is free), contradicting our choice of irreducible component of $U_{1,k}/H.$ Let $\sigma_1$ and $\sigma_2$ be one-dimensional cones in our cone decomposition whose classes in $I^+/Y$ correspond to the two irreducible components of $U_{1,k}$ of the pre-image. If $(\ell,1)\in X^\vee\oplus \Z$ is a primitive element of $\sigma_1$, we may without loss of generality assume that $(-\ell, 1)$ is a primitive element of $\sigma_2$. Because the intersection of the two irreducible components of the pre-image is non-trivial, it follows that
$$\tau:= \mathbf{R}_{\geq 0} (\ell, 1) +\mathbf{R}_{\geq 0}(-\ell +b(y,-),1)$$ is contained in the cone decomposition $\{\sigma_{\alpha}\}_{\alpha\in I}$ for some $y\in Y.$ This is true because by \cite{Kü}, Theorem 3.5(iv), there must be a cone $\tau$ in the cone decomposition such that both $\mathbf{R}_{\geq 0} (\ell, 1)$ and $\mathbf{R}_{\geq 0}(-\ell +b(y,-),1)$ (for some $y\in Y$) are faces of $\tau.$ Since the cone decomposition is smooth, the claim follows. A simple calculation shows that 
$$S_{(-y, \mathrm{Id})}(\tau)=S_{(0,[-1])}(\tau).$$ However, we already know from Lemma \ref{conefreelemma} that this is impossible. Hence all irreducible components of $U_{1,k}/H$ are smooth over $k$, which implies the last claim. All that remains to be shown is that $\mathscr{X}$ satisfies condition (iii) of Definition \ref{Kulikovdefinition}. By Proposition \ref{codim2proposition}, it suffices to exhibit a no-where vanishing global 2-form on $U_0/H.$ Consider the diagram
$$\begin{CD}
\mathrm{Bl}_{A[2]} A@>{\rho_\eta}>> X\\
@V{\pi_\eta}VV\\
A.
\end{CD}$$
By \cite{B}, Theorem Chapter 10, Theorem 10.6, we know that there exist no-where vanishing global 2-forms $\omega$ on $A$ and $\beta$ on $X$ such that $\pi_\eta^\ast\omega=\rho_\eta^\ast\beta.$ After replacing $\omega$ and $\beta$ by a scalar multiple if necessary, we may assume that $\omega$ extends to $\mathscr{A}.$ Now let $\pi\colon U_0\to \mathscr{A}$ be the canonical map, and let $\rho\colon U_0\to U_0/H$ be the quotient morphism. The 2-form $\pi^\ast\omega$ is $H$-invariant since this holds generically. 
Furthermore, $\pi^\ast \omega$ vanishes along the ramification locus of $\rho.$ Since $\rho$ is tamely ramified, we deduce that there is a global 2-form on $U_0/H$ which pulls back to $\pi^\ast\omega$ via $\rho$ and which restricts to $\beta$ at the generic fibre. By abuse of notation, we shall call this 2-form $\beta$ as well. We shall now show that $\beta$ vanishes no-where on $U_0/H.$ Indeed, we already know that $\beta$ does not vanish at the generic fibre. Furthermore, we know that $\beta$ does not vanish away from the branch locus of $\rho$, since $\pi^\ast \omega$ does not vanish away from the ramification locus of $\rho.$ But the complement of the intersection of the branch locus of $\rho$ and the special fibre of $U_0/H$ has codimension $\geq 2$ in $U_0/H$, which implies that $\beta$ does not vanish anywhere. 
\end{proof}\\
$\mathbf{Remark.}$ (i) This Theorem implies in particular that Kummer surfaces defined over strictly Henselian complete discrete valuation fields of residue characteristic $p>2$ potentially admit strictly semistable reduction in the category of schemes.\\
(ii) The finite extension from the Theorem can be chosen to be separable: Beginning from an Abelian surface over an arbitrary $K$, we first pass to a finite extension of $K$ to ensure that $A$ has semiabelian reduction and that $A[2]$ is constant. It follows from Grothendieck's $\ell$-adic monodromy theorem, and the fact that the endomorphism $[2]$ of $A$ is étale, that we can choose this extension to be separable. Finally, we pass to the finite extension from Proposition \ref{semconeexistenceproposition}. However, the proof of this Proposition shows that we only need to make sure that this last extension is of ramification index $\nu,$ and the extension $K[X]/\langle X^\nu+\pi_KX+\pi_K\rangle$ is a separable extension of $K$ with this property.\\
(iii) Semistable reduction of K3 surfaces (away from characteristic zero) is also addressed in \cite{Mau}, Section 4, as well as in \cite{LM}. In the first paper, potential semistable reduction of K3 surfaces is proved under the assumption that there exist a very ample line bundle on the K3 surface with small self-intersection compared to the residue characteristic. This condition seems difficult to establish for general Kummer surfaces, so those results are not directly applicable to our situation.\\
(iv) There are various higher-dimensional analogues of Kummer surfaces, all of which are normally called called \it Kummer varieties. \rm The straightforward way of generalizing Kummer surfaces to higher dimensions is to consider minimal desingularizations of the quotients $A/[-1]$ for general Abelian varieties $A$ of dimension $g\geq 2.$ Their arithmetic properties were studied, for example, in \cite{SkZ}. It seems reasonable to expect that the construction of models of Kummer surfaces presented in this paper generalizes to such Kummer varieties. However, these varieties are not Calabi-Yau if $g>2,$ so our definition of Kulikov models (generalized to higher dimensions in the obvious way) does not apply in this setup. Other generalizations of Kummer surfaces are more intricate, often involving Hilbert schemes. 
\section{Comparing degenerations of Abelian surfaces and their associated Kummer surfaces}
\subsection{Reduction and Galois representations}
Now that we have established the existence of strict Kulikov models of Kummer surfaces, at least after replacing the ground field by one of its finite extensions, we shall proceed to studying the relationship between the degeneration of Abelian surfaces and their associated Kummer surfaces. As before, let $A$ be an Abelian surface over $K$. For the moment, we make no further assumptions on $A$. It is possible to relate the étale cohomology of $A$ to the étale cohomology of $X$ as follows: Let $\ell$ be a prime number; we allow $\ell=p$. Choose (and fix) a separable closure $\overline{K}$ of $K$, and define a Galois representation
$$W_\ell:=\bigoplus_{\alpha\in A[2](\overline{K})} \Q_\ell\langle\alpha\rangle,$$ where $\Gal(\overline{K}/K)$ operates by permuting the basis elements. The following Lemma is certainly well-known to the experts, but it seems to be difficult to find a complete proof in the literature, so we provide one here for the sake of completeness:
\begin{lemma}
There is an isomorphism
\begin{align}H^2_{\et}(X_{\overline{K}}, \Q_\ell)=\bigwedge\nolimits^2 H^1_{\et}(A_{\overline{K}}, \Q_\ell)\oplus W_\ell(-1),\label{cohomologyformula}\end{align} which is  $\Gal(\overline{K}/K)$-equivariant. \label{Kummercohomologylemma}
\end{lemma}
\begin{proof}
Let $\tilde{A}:=\mathrm{Bl}_{A[2]} A,$ and let $\pi\colon \tilde{A}\to A$ be the canonical morphism. Because $\pi$ is surjective, the induced homomorphism of $\Q_\ell$-vector spaces
$H^2_{\et}(A_{\overline{K}}, \Q_\ell)\to H_{\et}^2(\tilde{A}_{\overline{K}}, \Q_\ell)$ is injective (\cite{Kl}, Proposition 1.2.4).
We also have the Chern class map 
$$\tilde{c}_1\colon \Pic \tilde{A}_{\overline{K}}\otimes_{\Z}\Z_\ell(-1)\to H_{\et}^2(\tilde{A}_{\overline{K}}, \Z_\ell),$$ which we shall now study in more detail. First recall that there is a canonical isomorphism $\Pic \tilde{A}_{\overline{K}}=\Pic A_{\overline{K}} \oplus \bigoplus_{\alpha\in A[2](\overline{K})} \Z\langle \alpha \rangle.$ The Kummer sequence $0\to \boldsymbol{\mu}_{\ell^n}\to \Gm\to \Gm\to 0$ of étale sheaves on $A_{\overline{K}}$ and $\tilde{A}_{\overline{K}}$ gives rise to a commutative diagram
\small
$$\begin{tikzcd}[cramped, sep=small, row sep=large]
0 \arrow[r] & \big(\varprojlim{\Pic A_{\overline{K}}}/\langle \ell^n \rangle\big)\otimes_{\Z_\ell}\Q_{\ell}(-1) \arrow[r] \arrow[d] & H_{\et}^2( A_{\overline{K}}, \Q_\ell) \arrow[r] \arrow[d]& (\varprojlim H_{\et}^2(A_{\overline{K}}, \Gm)[\ell^n])\otimes_{\Z_\ell}\Q_\ell(-1) \arrow[r] \arrow{d}{\cong}& 0\\
0\arrow[r]&\big(\varprojlim{\Pic \tilde{A}_{\overline{K}}}/\langle \ell^n \rangle\big)\otimes_{\Z_\ell}\Q_{\ell}(-1)\arrow[r] & H_{\et}^2(\tilde{A}_{\overline{K}}, \Q_\ell) \arrow[r] &(\varprojlim H_{\et}^2(\tilde{A}_{\overline{K}}, \Gm)[\ell^n])\otimes_{\Z_\ell}\Q_\ell(-1) \arrow[r] & 0
\end{tikzcd}$$
\normalsize
with exact rows. The second map in the bottom row induces $\tilde{c}_1\otimes \mathrm{Id}_{\Q_\ell}$, and the map on the right is an isomorphism because the cohomological Brauer group is a birational invariant for smooth projective surfaces (see, for example, \cite{DF}, Proposition 5). We find in particular that $W_\ell(-1)$ is contained in $H_{\et}^2(\tilde{A}_{\overline{K}}, \Q_\ell).$ A diagram chasing argument shows that $W_\ell(-1)$ and $H^2_{\et}(A_{\overline{K}}, \Q_\ell)$ intersect trivially in $H^2_{\et}(\tilde{A}_{\overline{K}}, \Q_\ell).$ We also deduce that the cokernel of $H_{\et}^2(A_{\overline{K}}, \Q_\ell)\to H_{\et}^2(\tilde{A}_{\overline{K}}, \Q_\ell)$ is isomorphic to $W_\ell(-1)$, which implies that $\dim_{\Q_\ell} H_{\et}^2(\tilde{A}_{\overline{K}}, \Q_\ell)=22.$ Since $\tilde{A}\to X$ is surjective, the map $H^2_{\et}(X_{\overline{K}}, \Q_\ell)\to H^2_{\et}(\tilde{A}_{\overline{K}}, \Q_p)$ is injective (again by \cite{Kl}, Proposition 1.2.4), and because $X$ is a K3 surface it must be an isomorphism. Since we already know that $H^2_{\et}(\tilde{A}_{\overline{K}}, \Q_\ell)\supseteq H^2_{\et}(A_{\overline{K}}, \Q_\ell)\oplus W_\ell(-1),$ and the space on the right also has dimension 22, this implies that 
$$H^2_{\et}(X_{\overline{K}}, \Q_\ell)=H^2_{\et}(A_{\overline{K}}, \Q_\ell)\oplus W_\ell(-1).$$
Together with the well-known fact that $H^2_{\et}(A_{\overline{K}}, \Q_\ell)=\bigwedge\nolimits^2 H^1_{\et}(A_{\overline{K}}, \Q_\ell)$, the claim of the Lemma follows.
\end{proof}\\
This calculation in étale cohomology immediately leads to the following partial converse to some of our previous results, for which we shall use the following (well-known) Lemma, the proof of which is a combination of results of C. Nakayama \cite{N}, and ideas going back to  Grothendieck \cite{SGA7}, I, and Rapoport-Zink \cite{RZ}:
\begin{lemma} Let $\ell\not=p$ be a prime number.\\
(i) Let $\mathscr{X}\to \Spec\Og_K$ be proper flat morphism from a regular scheme $\mathscr{X}.$ Suppose further that the special fibre of this morphism is a reduced divisor  \label{wellknownlemma} with normal crossings on $\mathscr{X}.$ Let $X$ be the generic fibre of $\mathscr{X}$. Then the wild inertia group $P\subseteq \Gal(\overline{K}/K)$ acts trivially on $H^i_{\et}(X_{\overline{K}}, \Q_\ell)$ for all $i\geq 0.$  Moreover, the operator on $H^i_{\et}(X_{\overline{K}}, \Q_\ell)$ induced by any $\sigma\in \Gal(\overline{K}/K)$ is unipotent. The same is true if $\mathscr{X}\to \Spec \Og_K$ is a strict Kulikov model of a smooth, projective, and geometrically integral surface $X$ over $K$ with trivial canonical bundle.\\
(ii) Let $A$ be an Abelian variety over $K$. Then $A$ has semiabelian reduction if and only if all $\sigma\in \Gal(\overline{K}/K)$ act unipotently on $H^1_{\et}(A_{\overline{K}}, \Q_\ell).$ 
\end{lemma}
\begin{proof}
Because of our assumptions on the reduction of $\mathscr{X}$ modulo $\mathfrak{m}_K,$ the first two claims of (i) follow from \cite{N}, Corollary 0.1.1 and Corollary 3.7. The third claim of (i) follows because by \cite{Ma2}, Proposition 2.3, there exists a weight spectral sequence for strict Kulikov models, so unipotence of the representation follows. Since the wild inertia subgroup $P\subseteq \Gal(\overline{K}/K)$ is a pro-$p$ group, the operator induced by any $g\in P$ on $H^i_{\et}(X_{\overline{K}}, \Q_\ell)$ is both unipotent and of finite order,  hence trivial. Part (ii) follows from \cite{SGA7}, IX, Corollaire 3.8. Note that the inertia subgroup $I$ of $\Gal(\overline{K}/K)$ coincides with all of $\Gal(\overline{K}/K)$ since $\Og_K$ is strictly Henselian.
\end{proof}
\begin{proposition}
Let $A$ be an Abelian surface over $K$ and let $X$ be the associated Kummer surface. Assume that there exists a (not necessarily strictly) semistable model $\mathscr{X}\to \Spec\Og_K$ of $X$ which is a scheme. Then $A[2]$ is a constant $K$-group scheme. The same is true if $\mathscr{X}\to \Spec \Og_K$ is a strict Kulikov model of $X$.  \label{converseproposition}
\end{proposition}
\begin{proof}
Since $X$ has semistable reduction and since $\Og_K$ is strictly Henselian, all $\sigma\in \Gal(\overline{K}/K)$ act unipotently on $H^2_{\et}(X_{\overline{K}}, \Q_\ell)$ (see Lemma \ref{wellknownlemma} (i)) and hence on $W_\ell$ (since $K$ is strictly Henselian and $\ell\not=p,$ the $\ell$-adic cyclotomic character is trivial). We must show that the action of any $\sigma\in \Gal(\overline{K}/K)$ on $W_\ell$ is trivial. Fix such a $\sigma$. Then the operator on $W_\ell$ induced by $\sigma$ both has finite order and is unipotent, so by looking at the minimal polynomial of $\sigma$ on $W_\ell$ one sees easily that $\sigma$ must act trivially. 
\end{proof}\\
It is not in general true that $A$ has semiabelian reduction if $X$ has strictly semistable reduction. However, this statement is true up to quadratic twist. For the proof of this result we will need the following 
\begin{lemma}
Let $V$ be a $4$-dimensional vector space over a field $k$ of characteristic 0 and let $f$ be a linear operator on $V$. \\
(i) Assume that the induced operator $\wedge^2f$ on $\bigwedge\nolimits^2 V$ is unipotent. If $f$ is not unipotent, then $-f$ is.\\
(ii) Suppose that $\wedge^2f$ equals the identity. Then $f=\pm\mathrm{Id}_V.$
\end{lemma}
\begin{proof}
(i) Since being unipotent is a condition on the characteristic polynomial of an operator, which is invariant under extensions of the ground field, we may assume without loss of generality that $k$ is algebraically closed. Suppose first that $f$ is diagonalizable with eigenvalues $\lambda_1,..., \lambda_4.$ Then, for all $1\leq i,j\leq 4$, $i<j$, $\lambda_i\lambda_j$ is an eigenvalue of $\wedge^2f.$ But since $\wedge^2f$ is unipotent, we deduce that all $\lambda_j$ must be equal, and that their common value must be either $1$ or $-1,$ so $f=\pm\mathrm{Id}_V.$ The case where $f$ has two Jordan blocks of size 1 and one Jordan block of size 2 can be dealt with analogously. Suppose now that $f$ has precisely two Jordan blocks, with eigenvalues $\lambda_1$ and $\lambda_2.$ Then there is a basis $e_1,..., e_4$ of $V$ such that $f(e_1)=\lambda_1e_1,$ $f(e_2)=\lambda_1e_2+e_1$, and $f(e_3)=\lambda_2e_3.$ A simple calculation shows that $\lambda_1^2$ and $\lambda_1\lambda_2$ are eigenvalues of $\wedge^2f.$ Again because $\wedge^2f$ is unipotent, it follows that $\lambda_1=\lambda_2$, and that these numbers must be either equal to $1$ or to $-1$. Finally suppose that $f$ only has one Jordan block with eigenvalue $\lambda.$ Let $e_1,...,e_4$ be a basis of $V$ such that $f(e_1)=\lambda e_1$ and $f(e_2)=\lambda e_2+e_1.$ Clearly, $\lambda^2$ is an eigenvalue of $\wedge^2f, $ so $\lambda=\pm1.$ Putting all pieces together, the claim follows. Statement (ii) can be proved in a way entirely analogous to part (i) and will be left to the reader. 
\end{proof}\\
For the next Proposition, recall that, given a continuous homomorphism $\Gal(\overline{K}/K)\to \{1, -1\},$ we can construct the \it quadratic twist \rm of $A$ by  $q$ as follows: We define an action of $\Gal(\overline{K}/K)$ $A_{\overline{K}}$ by declaring that $\sigma\in \Gal(\overline{K}/K)$ act on a functorial point $x$ of $A$ as $x\mapsto q(\sigma) x^{\sigma}.$ Here $-^\sigma$ refers to the Galois action on $A_{\overline{K}}=A\times_K\Spec\overline{K}$ on the second factor. The quotient of this action is an Abelian variety over $K,$ which we shall denote by $A^q.$ Clearly, $A_{\overline{K}}$ and $A^q_{\overline{K}}$ are canonically isomorphic, and we obtain an isomorphism between their associated Kummer surfaces. It follows immediately from the definitions that the isomorphism between the Kummer surfaces is Galois equivariant, and hence descends to $K$. Therefore $A$ and $A^q$ define the same Kummer surface.
\begin{proposition}
Let $A$ be an Abelian surface over $K$ and assume that the Kummer surface $X$ associated with $A$ has semistable reduction or admits a strict Kulikov model. Then there exists a continuous homomorphism \label{quadratictwistproposition}
$q\colon \Gal(\overline{K}/K)\to \{1,-1\}$ such that the quadratic twist $A^q$ of $A$ by $q$ has semiabelian reduction.
\end{proposition}
\begin{proof}
Denote by 
$$\rho\colon \Gal(\overline{K}/K)\to \mathrm{Aut}_{\Q_\ell}(H^1_{\et}(A_{\overline{K}}, \Q_\ell))$$ the Galois representation attached to $A$.  By Lemma \ref{wellknownlemma} (i), we know that the action of any $\sigma\in \Gal(\overline{K}/K)$ on $H^2_{\et}(X_{\overline{K}}, \Q_\ell)$ is unipotent. It follows from Formula (\ref{cohomologyformula}) (Lemma \ref{Kummercohomologylemma}) that all $\sigma\in \Gal(\overline{K}/K)$ operate unipotently on $\bigwedge^2 H^1_{\et}(A_{\overline{K}}, \Q_\ell).$ From the Lemma preceding this Proposition we may deduce that, for $\sigma\in \Gal(\overline{K}/K)$, either $\rho(\sigma)$ or $-\rho(\sigma)$ is unipotent. Using Lemma \ref{wellknownlemma} (i), we further deduce that $\rho(g)=\pm \mathrm{Id}$ if $g$ lies in the wild inertia subgroup $P\subseteq \Gal(\overline{K}/K).$ We shall first prove that the image of $\rho$ is Abelian. Let $\sigma, \tau\in \Gal(\overline{K}/K).$ Since $\Gal(\overline{K}/K)/P$ is Abelian, there exists $g\in P$ such that $\sigma\tau=g\tau\sigma.$ If $\rho(g)$ were equal to $-\mathrm{Id},$ we would have
$$\rho(\sigma)\rho(\tau)=-\rho(\tau)\rho(\sigma).$$
This, however, would mean that $\rho(\sigma)\rho(\tau)\rho(\sigma)^{-1}=-\rho(\tau),$ so
$$\mathrm{trace} \,\rho(\tau)=\mathrm{trace} (-\rho(\tau)).$$ But this is impossible since $\mathrm{trace}\,\rho(\tau)=\pm4.$ Hence we must have $\rho(g)=\mathrm{Id},$ which implies $\rho(\sigma)\rho(\tau)=\rho(\tau)\rho(\sigma).$ Using the preceding Lemma, we now define, for all $\sigma\in \Gal(\overline{K}/K),$
$$q(\sigma):=\begin{cases} 1 & \text{if $\sigma$ acts unipotently on $H^1_{\et}(A_{\overline{K}}, \Q_\ell)$}\\
                                           -1 & \text{if $-\sigma$ acts unipotently on $H^1_{\et}(A_{\overline{K}}, \Q_\ell)$}.
                       \end{cases}$$
First observe that $q$ is indeed a homomorphism. This can be seen as follows: By what we have just proved, the operators defined by any $\sigma, \tau\in \Gal(\overline{K}/K)$ on $ H^1_{\et}(A_{\overline{K}}, \Q_\ell)$ commute. Hence, we can find a basis of $H^1_{\et}(A_{\overline{K}}, \Q_\ell)$ with respect to which both these operators have upper triangular form, and such that all entries on the diagonal are $q(\sigma)$, $q(\tau),$ respectively. The product of these two matrices will also have upper triangular form, with all diagonal entries equal to $q(\sigma)q(\tau).$ Hence $q$ is a homomorphism. Let us now show that $q$ is continuous. Let $\Q_\ell[x]^{(4)}$ be the vector space of polynomials of degree $\leq 4$ over $\Q_\ell,$ endowed with the topology inherited from $\Q_\ell.$ Then the map 
$$\mathrm{End}_{\Q_\ell}(H^1_{\et}(A_{\overline{K}}, \Q_\ell))\to \Q_\ell[x]^{(4)}$$ which maps an operator to its characteristic polynomial is continuous. Since Galois representations on étale cohomology spaces are always continuous, we see that the homomorphism $\Gal(\overline{K}/K)\to \Q_\ell[x]^{(4)}$ sending $\sigma\in \Gal(\overline{K}/K)$ to the characteristic polynomial of the operator $\rho(\sigma)$ on $H^1_{\et}(A_{\overline{K}}, \Q_\ell)$ is also continuous. However, we know from the Lemma preceding this Proposition that the characteristic polynomial of the operator on $H^1_{\et}(A_{\overline{K}}, \Q_\ell)$ defined by $\sigma\in \Gal(\overline{K}/K)$ will be either $(x-1)^4$ (if $\sigma$ acts unipotently) or $(x+1)^4$ (if $-\sigma$ acts unipotently). This implies that $q$ is continuous. In particular, we can construct the quadratic twist $A^q$ of $A$ by $q$. By construction, the $\Q_\ell$-vector spaces $H^1_{\et}(A_{\overline{K}}, \Q_\ell)$ and $H^1_{\et}(A^q_{\overline{K}}, \Q_\ell)$ are canonically identified, and the action of $\sigma\in \Gal(\overline{K}/K)$ on $H^1_{\et}(A^q_{\overline{K}}, \Q_\ell)$ is equal to $q(\sigma)$ times the action of $\sigma$ on $H^1_{\et}(A_{\overline{K}}, \Q_\ell)$. Hence $\Gal(\overline{K}/K)$ acts unipotently on $H^1_{\et}(A^q_{\overline{K}}, \Q_\ell).$ By Lemma \ref{wellknownlemma} (ii), it follows that $A^q$ has semiabelian reduction. 
\end{proof}\\
\noindent $\mathbf{Remark}.$ (i) The Abelian surface $A$ will always be tamely ramified (i.e., the wild inertia group will act trivially on $H^1_{\et}(A_{\overline{K}}, \Q_\ell))$ provided that $p=0$ or $p>5$ (see \cite{L}, Theorem 3.9 for the case where $p>0$). Hence one can give a shorter proof of the fact that $\im \rho$ is Abelian in these cases.\\
(ii) In the situation above (i.e., if the Kummer surface associated with $A$ has semistable reduction), it follows in particular that there is a \it unique \rm quadratic twist of $A$ which has semiabelian reduction. Indeed, if $A_1$ and $A_2$ are both quadratic twists of $A$ with semiabelian reduction, then $A_2$ is a quadratic twist of $A_1$ by some continuous character $q\colon \Gal(\overline{K}/K)\to \{1,-1\}.$ The $\Q_\ell$-vector spaces $H^1_{\et} (A_{1, \overline{K}}, \Q_{\ell})$ and $H^1_{\et}(A_{2, \overline{K}}, \Q_{\ell})$ are canonically identified, and the $\Gal(\overline{K}/K)$-actions differ by precisely $q.$ Hence, in order for both Galois representations to be unipotent, it is necessary that $q$ be trivial, so $A_1=A_2.$ This is no longer true if we remove the hypothesis that $\Og_K$ be strictly Henselian. If $K$ admits non-trivial unramified quadratic extensions, we could twist by a non-trivial unramified quadratic character without affecting the Abelian surfaces' reduction behaviour.\\
\\
Let $A$ be an Abelian variety over $K$ with semiabelian reduction. As indicated above, the $\Q_\ell$-vector space $H^1_{\et}(A_{\overline{K}}, \Q_\ell)$ comes with a nilpotent monodromy operator $N$. The following (well-known) Proposition shows that $N$ contains much information about the reduction of $A$:
\begin{proposition}
Let $A$ be an Abelian variety over $K$ with semiabelian reduction and let $\mathscr{A}\to \Spec\Og_K$ be its Néron model. Let $N$ be the monodromy operator on $H^1_{\et}(A_{\overline{K}}, \Q_\ell)$. Then we have
$$t(A)=r_2=\mathrm{rank}_{\Q_\ell} N.$$

\end{proposition}
\begin{proof}
This follows from Grothendieck's orthogonality theorem. More precisely, let $A^\vee$ denote the dual Abelian variety of $A$ and recall that, if $\sigma$ is a topological generator of the image of the Galois representation on $T_\ell(A)\otimes_{\Z_\ell}\Q_\ell,$ we have $(\sigma-1)^2=0$ (by \cite{SGA7}, IX, Corollaire 3.5.2), so $N=\sigma-1.$ Consider the filtration 
$$0\subseteq T_\ell(A)^t\subseteq T_\ell(A)^{I}\subseteq T_\ell(A),$$ where $T_\ell(A)^{I}$ stands for the $\Z_\ell$-sublattice of $\Gal(\overline{K}/K)$-invariant elements (which therefore, modulo $\ell^n$, extend to sections of the Néron model $\mathscr{A}$ of $A$), and $T_\ell(A)^t$ stands for the sublattice of $T_\ell(A)^{I}$ consisting of all elements which, modulo $\ell^n$, restrict to the toric part of $\mathscr{A}^0_k$. We shall employ analogous notation for $T_\ell(A^\vee).$ By \cite{SGA7}, IX, Théorème 5.2, $T_\ell(A^\vee)^{I}$ is the orthogonal complement of $T_\ell(A)^t$ with respect to the Weil pairing. In particular, we see that the image of the monodromy operator $N$ on $T_\ell(A)\otimes_{\Z_\ell}\Q_\ell$ is contained in $T_\ell(A)^t\otimes_{\Z_\ell}\Q_\ell,$ and we have $\ker N=T_\ell(A)^{I}.$ Furthermore, the Weil pairing induces a surjection
$$T_\ell (A)\otimes_{\Z_\ell}\Q_\ell\overset{\cong}{\to} \Hom_{\Q_\ell}(T_\ell(A^\vee)\otimes_{\Z_\ell}\Q_\ell, \Q_\ell(1)) \to \Hom_{\Q_{\ell}}(T_\ell(A^\vee)^t\otimes_{\Z_\ell}\Q_\ell, \Q_\ell(1))$$ whose kernel is precisely $T_\ell(A)^{I}\otimes_{\Z_\ell}\Q_\ell.$ But since $t(A)=\rk_{\Z_\ell} T_\ell(A)^t=\rk_{\Z_\ell} T_\ell(A^\vee)^t,$ a dimension counting argument implies the claim. 
\end{proof}
\subsection{Relations between the degenerations}
Let $A$ be an Abelian surface over $K$ and let $X$ be the associated Kummer surface. Assume that $X$ admits a strict Kulikov model $\mathscr{X}\to \Spec\Og_K$. The aim of the present subsection is to prove that the degeneration behaviour of the Kummer surface of $X$ is completely governed by the degeneration behaviour of $A$. Let $\mathscr{A}$ be the Néron model of $A$ over $\Og_K$ and recall that there is a nonnegative integer $r=r_2$ such that we have an exact sequence
$$0\to \mathbf{G}_{\mathrm{m}}^r\to \mathscr{A}^0_k\to B\to 0,$$ where $B$ is an Abelian variety over $k$. We have the following
\begin{theorem}
Let $A$ be an Abelian surface over $K$ with semiabelian reduction. Let $X$ be the associated Kummer surface. Assume that $X$ admits a strict Kulikov model $\mathscr{X}\to \Spec\Og_K.$ Then the special fibre $\mathscr{X}_k$ is of type I (resp. type II, type III) if and only if the toric rank $r=t(A)$ of $A$ is equal to $0$ (resp. 1, 2). 
\end{theorem} 
\begin{proof}
Since $A$ has semiabelian reduction, we know that $H^2_{\et}(A_{\overline{K}}, \Q_\ell)$ is a unipotent representation. By Formula (\ref{cohomologyformula}) from Lemma \ref{Kummercohomologylemma}, we have
$$H^2_{\et}(X_{\overline{K}}, \Q_\ell)=\bigwedge\nolimits^2 H^1_{\et}(A_{\overline{K}}, \Q_\ell)\oplus W_\ell(-1),$$ on which the monodromy operator $N_X$ is given by
\begin{align}N_X= (N\wedge \mathrm{Id}+\mathrm{Id}\wedge N) \oplus 0,\label{Nformula}\end{align} where $N$ denotes the monodromy operator on $H^1_{\et}(A_{\overline{K}}, \Q_\ell)$. This can be seen as follows: Form the proof of Proposition \ref{converseproposition}, we know that the Galois representation on $W_\ell(-1)$ is trivial, which means that the restriction of $N_X$ to $W_\ell(-1)$ vanishes. One also sees easily that, on $H^1_{\et}(A_{\overline{K}}, \Q_\ell)^{\otimes 2}$, we have $\log(\sigma\otimes\sigma)=\log((\sigma\otimes 1)(1\otimes \sigma))=(\log\sigma)\otimes1+1\otimes\log\sigma.$ Hence formula (\ref{Nformula}) follows by considering the surjection $H^1_{\et}(A_{\overline{K}}, \Q_\ell)^{\otimes 2}\to\bigwedge\nolimits^2 H^1_{\et}(A_{\overline{K}}, \Q_\ell).$ Let $e_1,..., e_4$ be a basis of $H^1_{\et}(A_{\overline{K}}, \Q_\ell)\otimes_{\Q_\ell} \overline{\Q}_\ell$ with respect to which $N$ has Jordan normal form. Clearly, if $N=0$ then $N_X=0$. Now suppose that $d=\mathrm{rank}_{\Q_\ell} N=1$. Then we may assume without loss of generality that $N(e_1)=N(e_3)=N(e_4)=0$ and $N(e_2)=e_1.$ In this case, direct calculation shows that $N_X(e_2\wedge e_3)=e_1\wedge e_3$ and $N_X(e_2\wedge e_4)=e_1\wedge e_4$ (so $N_X\not=0$), but that $N_X(e_i\wedge e_j)=0$ in all other cases ($i<j$), so that $N_X^2=0$. Finally, assume that $d=\mathrm{rank}_{\Q_\ell} N=2$. In this case, $N$ has two Jordan blocks of size $2$ (since $N=\sigma-1$ for some topological generator $\sigma$ of the monodromy group, and $(\sigma-1)^2=0$). Again, we can write down $N_X$ explicitly in terms of the induced basis of $\bigwedge^2 H^1_{\et}(A_{\overline{K}}, \Q_\ell)$ and conclude that $N_X^2\not=0$ but $N_X^3=0.$ The calculations will be left to the reader. 
\end{proof}\\
We can now proceed to studying the relationship between the dual complex of a strict Kulikov model of an Abelian surface $A$ and that of the strict Kulikov model of the associated Kummer surface $X$. Suppose that $A$ has semi-Abelian reduction and that $A[2]$ is a constant group scheme over $K$. Let $(G,\mathscr{L}, \mathscr{M})\in \mathrm{DEG}^{\mathrm{split}}_{\mathrm{ample}}$ with $G=\mathscr{A}^0,$ and suppose that the finite group $H=\{\mathrm{Id}, [-1]\}$ operates on this object. Suppose further that $(X,Y,\phi,a,b)$ is the object of $\mathcal{C}$ associated with $(G, \mathscr{L}, \mathscr{M}),$ and assume that there is a smooth $\Gamma$-admissible rational polyhedral cone decomposition $\{\sigma_{\alpha}\}_{\alpha\in I}$ of $\mathscr{C}\subseteq X^\vee_{\mathbf{R}}\oplus \mathbf{R}$ which satisfies properties (a),..., (d) of Theorem \ref{Pexistencetheorem}. By Theorem \ref{Pexistencetheorem} and Corollary \ref{Pkulikovcorollary}, we know that $A$ admits a strict Kulikov model $P$ which is a scheme and which contains the Néron model $\mathscr{A}$ of $A$ as an open subscheme. Furthermore, we know from Theorem \ref{Kummersemistabletheorem} and its proof that $\mathscr{X}=(\mathrm{Bl}_{\mathscr{A}[2]}P)/[-1]$ is a strict Kulikov model of $X.$ This implies that there is a close relationship between the dual complexes of $P$ and $\mathscr{X}$ (and in particular between the numbers of irreducible components of the special fibres of $P$ and $\mathscr{X}):$
\begin{theorem}
Let $A$ be an Abelian surface with strict Kulikov model $P$, as at the beginning of this section. Let $\mathscr{X}$ be the associated strict Kulikov model of the Kummer surface $X$. Let $\Delta_A$, $\Delta_X$ be the dual complexes of the special fibres of $P$ and $\mathscr{X},$ respectively. Then the finite group $H=\{\mathrm{Id}, [-1]\}$ acts on $\Delta_A,$ and the quotient is canonically isomorphic to $\Delta_X.$
\end{theorem}
\begin{proof}
This is an immediate consequence of the construction of $\mathscr{X},$ and the fact that the canonical map $\mathrm{Bl}_{\mathscr{A}[2]} P\to P$ induces an isomorphism of dual complexes of special fibres.
\end{proof}\\
Now we would like to understand the relationship between the number of irreducible components of special fibres of Néron models and strict Kulikov models. 
\begin{lemma}
Let $A$ be an Abelian surface over $K$ with semiabelian reduction and such that the $K$-group scheme $A[2]$ is constant. Then the group $\Phi$ of connected components of the special fibre $\mathscr{A}_k$ of the Néron model $\mathscr{A}$ of $A$ can be written as
$$\Phi\cong \bigoplus_{i=1}^{t} \Z/d_i\Z$$ such that the integers $d_j$ are even. Here, $t=t(A)$ denotes the toric rank of $A$. 
\end{lemma}
\begin{proof}
Choose a split ample degeneration $(G, \mathscr{L}, \mathscr{M})$ with $G=\mathscr{A}^0$ as before, and let $(X,Y,\phi, a,b)=\mathrm{For}(F((G, \mathscr{L}, \mathscr{M}))).$ The pairing $b\colon Y\times X\to \Z$ induces an injective map $Y\to X^\vee$. By \cite{FC}, Chapter III, Corollary 8.2, we know that $\Phi$ is isomorphic to the cokernel of this map. Now choose a basis $e_1,..., e_t$ of $X^\vee$ such that there exist non-zero integers $\lambda_1,..., \lambda_t$ with the property that $\lambda_1e_1,..., \lambda_te_t$ is a basis of the image of $Y\to X^\vee.$ We know from Proposition \ref{evenproposition} that the image of any element $y\in Y$ in $X$ is of the form $2f$ for some $f\in X^\vee.$ In particular, the $\lambda_je_j$ are of this form, which implies that the integers $\lambda_j$ are even. 
\end{proof}\\
\noindent$\mathbf{Remark.}$ This Lemma, together with Proposition \ref{converseproposition}, implies that requiring the existence of a semistable model $\mathscr{X}\to \Spec\Og_K$ of the Kummer surface $X$ associated with an Abelian surface $A$ is a stronger condition on $A$ than one might initially think. For example, suppose that, for $i=1,2,$ $E_i$ is an elliptic curve over $K$ which has either good reduction or is of type $I_1$ (in Kodaira's notation), and such that at least one of the $E_j$ has bad reduction. Then, if $A:=E_1\times_KE_2,$ $X$ does not admit a semistable model.
\begin{proposition}
Let $\Phi$ denote the group of irreducible components of $\mathscr{A}_k.$ Then the number $N_X$ of irreducible components of the special fibre of $\mathscr{X}$ is equal to 
$$N_X=\#\Phi[2]+\frac{1}{2}\#(\Phi\backslash \Phi[2])=\frac{1}{2}\#\Phi+2^{t-1},$$ where $t=t(A)$ is the toric rank of $A$. \label{Phiproposition}
\end{proposition}
\begin{proof}
This follows because the morphism $\mathrm{Bl}_{\mathscr{A}[2]}P\to P$ induces a bijection between the sets of irreducible components of the special fibres of those schemes (this follows from Lemma \ref{basechangelemma} together with the fact that the centre of the blow-up is contained in the smooth locus of $P$) which is equivariant with respect to the operations of $H$ on both schemes. Furthermore, we use the fact that the open immersion $\mathscr{A}\to P$ also induces an $H$-equivariant bijection between sets of irreducible components of special fibres. By the Lemma preceding this Proposition, we find that $\#\Phi[2]=2^t,$ so the claim follows. 
\end{proof}
\section{Strict Kulikov models and base change}
Let $L/K$ be a finite extension and let $\Og_L$ be the ring of integers of $L$. Then $\Og_L$ is a strictly Henselian discrete valuation ring, and if $\mathfrak{m}_L$ denotes its maximal ideal, the canonical morphism $k\to \Og_L/\mathfrak{m}_L$ is an isomorphism. Suppose that $A$ is an Abelian surface over $K$ with Néron model $\mathscr{A}$ and associated Kummer surface $X$. Assume (as always) that $A$ has semiabelian reduction and that $A[2]$ is a constant $K$-group scheme. Let $(G, \mathscr{L}, \mathscr{M})$ be a split ample degeneration over $K$ with $G=\mathscr{A}^0$ as before. Further suppose that $(X,Y,  \phi, a, b):=\mathrm{For}(F((G, \mathscr{L}, \mathscr{M})))$ as in Section \ref{Prelimsection} and that there exists a smooth rational polyhedral cone decomposition of $\mathscr{C}\subseteq X^\vee\oplus \Z$ satisfying properties (a),..., (d) from Theorem \ref{Pexistencetheorem}. As in the case of Néron models, taking (strict) Kulikov models does not in general commute with base change. If $\mathscr{A}_L$ denotes the Néron model of $A_L:=A\times_K\Spec L$ over $\Og_L$ then the morphism $\mathscr{A}^0\times_{\Og_K}\Spec\Og_L\to \mathscr{A}_L^0$ is an isomorphism (because $A$ has semiabelian reduction), but the morphism $\mathscr{A}\times_{\Og_K}\Spec\Og_L\to \mathscr{A}_L$ is not in general surjective. In other words, extending the ground field from $K$ to $L$ leads to a Néron model over $\Og_L$ whose special fibre has more irreducible components than the special fibre of that over $K.$ If $\Phi$ denotes the group of connected components of (the special fibre of) $\mathscr{A}$ and $\Phi_L$ that of $L$, then $$\#\Phi_L=e_{L/K}^t\#\Phi,$$ where $t=t(A)$ denotes the toric rank of $A$ and $e_{L/K}$ denotes the index of ramification of the extension $L/K$. This follows, for example, from \cite{FC}, Chapter III, Corollary 8.2. Now let $\mathscr{X}$ be the strict Kulikov model of $X$ from Theorem \ref{Kummersemistabletheorem}. The aim of the present section is to understand how the number of irreducible components of strict Kulikov models of Kummer surfaces changes under base change. More precisely, assume that there exists a smooth rational polyhedral cone decomposition of $\mathscr{C}'\subseteq (X')^\vee_{\mathbf{R}}\oplus \mathbf{R}$ which satisfies the conditions (a),..., (d) from Theorem \ref{Pexistencetheorem}. Here, we denote $\mathrm{For}(F((G, \mathscr{L}, \mathscr{M})\times_K\Spec L))$ by $(X', Y', \varphi', a', b').$ This will always be the case after replacing $L$ by one of its finite extensions by Proposition \ref{semconeexistenceproposition}. We can now prove
\begin{theorem}
Keep the notation from the beginning of this section. Assume that there exists a smooth rational polyhedral cone decomposition of $\mathscr{C}'\subseteq (X')^\vee_{\mathbf{R}}\oplus \mathbf{R}$ which satisfies the conditions (a),..., (d) from Theorem \ref{Pexistencetheorem}, so that the Kummer surface $X_L$ admits a strict Kulikov model $\mathscr{X}_L$ over $\Og_L$, as constructed in Theorem \ref{Kummersemistabletheorem}. If $N$ and $N_L$ denote the number of irreducible components of $\mathscr{X}_k$ and $\mathscr{X}_{L,k}$, respectively, the formula 
$$N_L=e_{L/K}^tN-2^{t-1}(e_{L/K}^t-1)$$ holds.
\end{theorem}
\begin{proof}
Let $\Phi_L$ be the group of connected components of the special fibre of the Néron model of $A\times_K\Spec L.$ By Proposition \ref{Phiproposition}, we know that 
\begin{align*}
N_L&=\frac{1}{2}\#\Phi_L+2^{t-1}\\
&=\frac{1}{2}e_{L/K}^t \#\Phi+2^{t-1}\\
&=e_{L/K}^t(\frac{1}{2}\#\Phi+2^{t-1})-2^{t-1}(e_{L/K}^t-1)\\
&=e_{L/K}^tN-2^{t-1}(e_{L/K}^t-1). 
\end{align*}
\end{proof}

\section{Equivariant Kulikov models of Kummer surfaces, and the monodromy conjecture}
In this section, we shall prove that Kummer surfaces in equal characteristic zero admit \it equivariant Kulikov models. \rm Together with previous work of Halle-Nicaise (\cite{HN3}, Corollary 5.3.3), this will imply that the monodromy conjecture is true for Kummer surfaces. Throughout this section (except for the final corollary), we shall assume that the residue field $k$ of $\Og_K$ is of characteristic zero.
\begin{definition} Let $X$ be a smooth, projective, geometrically integral surface over $K$ with trivial canonical bundle, and let $d$ be a positive integer. An \rm equivariant Kulikov model \it of $X$ over $\Og_{K(d)}$ is a Kulikov model $\mathscr{X}(d)\to \Spec\Og_{K(d)}$ of $X(d):=X\times_K\Spec K(d)$ with the property that the action of $\Gal(K(d)/K)=\boldsymbol{\mu}_d$ extends to $\mathscr{X}(d).$
\end{definition}
For a precise statement of the monodromy conjecture (or rather a refined version thereof), see \cite{HN3}, Definition 2.3.5. Roughly speaking, it can be summarized as follows: If $X$ is a smooth, projective, geometrically integral algebraic variety over $K$ with trivial canonical bundle (generated by a global top-form which we call $\omega$), we can consider the \it motivic Zeta function $Z_{X,\omega}(t),$ \rm which is an element of the ring $\mathcal{M}^{\widehat{\boldsymbol{\mu}}}_k[\![ t ]\!]$, where $\mathcal{M}^{\widehat{\boldsymbol{\mu}}}_k:=K_0^{\widehat{\boldsymbol{\mu}}}(\mathrm{Var}_k)[\boldsymbol{\mathrm{L}}^{-1}].$ Here, $K_0^{\widehat{\boldsymbol{\mu}}}(\mathrm{Var}_k)$ denotes the \it $\widehat{\boldsymbol{\mu}}-$equivariant Grothendieck ring of varieties; \rm see \cite{HN3},  (2.2.1) and Definition 2.3.1 of \it loc. cit. \rm for more details. We shall always use the notation
$$\widehat{\boldsymbol{\mu}}:=\varprojlim_{d\in \N} \boldsymbol{\mu}_d=\Gal(\overline{K}/K),$$ with the indices $d$ ordered by divisibility. Let $\sigma$ be a topological generator of this group. Then $\sigma$ acts on the $\Q_\ell$-vector space $H_{\et}^i(X_{\overline{K}}, \mathbf{Q_\ell}).$ The monodromy conjecture now asserts that $Z_{X, \omega}(t)$ can be written as a polynomial in $t, \Big(\frac{1}{1-\mathbf{L}^at^b}\Big)_{(a,b)\in S\subseteq \Z\times\Z_{>0}}$ with $S$ finite and such that for all $(a,b)\in S$, there is some $i\geq 0$ such that $\exp({2\pi \sqrt{-1}\frac{a}{b}})$ is an eigenvalue of the action of $\sigma$ on $H^i(A_{\overline{K}}, \Q_\ell)$ with respect to any embedding $\Q_\ell\to \mathbf{C}.$ 
\begin{theorem} Let $A$ be an Abelian surface over $K$ with associated Kummer surface $X$. Then there exists $d_0\in \N$ such that $X$ admits an equivariant Kulikov model for all $d_0\mid d.$ \label{equivariantKulikovtheorem}
\end{theorem}
\begin{proof}
Since the residue field $k$ is of characteristic zero, all finite extensions of $K$ are of the form $K(d)$ for some positive integer $d$. Choose $d>0$ with the property that $A(d):=A\times_K\Spec K(d)$ has semiabelian reduction and that its 2-torsion is constant over $K(d)$. Let us identify $\Gal(K(d)/K)$ with $\boldsymbol{\mu}_d.$ 
The action of $\boldsymbol{\mu}_d$ on $K(d)$ induces an action on $\Og_{K(d)}$. Furthermore, the canonical action of $\boldsymbol{\mu}_d$ on $A(d)$ extends uniquely to an action of $\boldsymbol{\mu}_d$ on the identity component $\mathscr{A}(d)^0$ of the Néron model $\mathscr{A}(d)$ of $A(d)$ over the action of $\boldsymbol{\mu}_d$ on $\Og_{K(d)}.$ This follows from the universal property of the Néron model. Observe also that the actions of $H=\{\mathrm{Id}, [-1]\}$ and $\boldsymbol{\mu}_d$ on $\mathscr{A}(d)^0$ commute (this holds generically since the action of $H$ is defined over $K$, so it must hold globally). In particular, we obtain an action of $H\times \boldsymbol{\mu}_d$ on $\mathscr{A}(d)^0.$ We also obtain an action of $H\times \boldsymbol{\mu}_d$ on $\Spec \Og_{K(d)}$ (via the second factor), and the two actions are compatible in the obvious way. Now choose a split ample degeneration $(G, \mathscr{L}, \mathscr{M})$ over $K(d)$ with $G=\mathscr{A}(d)^0$. Replacing $\mathscr{L}$ by $\otimes_{(h, \tau)\in H\times\boldsymbol{\mu}_d} (h, \tau)^\ast\mathscr{L}$ (and similarly for $\mathscr{M}$), we may assume that $H\times\boldsymbol{\mu}_d$ acts on the object $(G, \mathscr{L}, \mathscr{M})\in \mathrm{DEG}^{\mathrm{split}}_{\mathrm{ample}}$ over the action of $H\times \boldsymbol{\mu}_d$ on $\Spec \Og_{K(d)}.$ Using \cite{Kü}, Proposition 3.3, we may deduce from \cite{Kü}, Theorem 3.5 that there exists a projective regular model $P$ of $A$ (depending on the choice of a suitable polyhedral cone decomposition $\{\sigma_{\alpha}\}_{\alpha\in I}$ as before) on which $H\times \boldsymbol{\mu}_d$ acts in a way compatible with the action of $\boldsymbol{\mu}_d$ on $\Spec\Og_{K(d)}.$ Furthermore, we know that the special fibre of $P$ is a (not necessarily reduced) divisor whose associated reduced divisor has strict normal crossings, and that the reduced special fibre $(P_k)_{\mathrm{red}}$ of $P$ has a stratification indexed by $I^+/Y$, where $I^+$ is the set of indices whose associated cone is positive-dimensional, and such that the action of $H\times \boldsymbol{\mu}_d$ on $P$ preserves the stratification. Furthermore, the action of $H\times \boldsymbol{\mu}_d$ on the set of strata is given by the action of this group on $I^+/Y.$ From Lemma \ref{conefreelemma}, we know that $H$ cannot fix any points on the special fibre which are not contained in a stratum associated with a cone of the form $\mathbf{R}_{\geq 0}(\ell, 1)$ for some $\ell\in X^\vee.$ But those strata are contained in the smooth locus of $P\to \Spec\Og_{K(d)}.$ By \cite{Kü}, 4.4, we know that the Néron model $\mathscr{A}$ of $A$ is contained in $P$, and by the argument from the proof of Theorem \ref{Pexistencetheorem}, we know that the induced morphism $\mathscr{A}\to P^{\mathrm{sm}}$ is an isomorphism. It now follows (from the same arguments as in the proof of Theorem \ref{Kummersemistabletheorem}) that $\mathscr{Y}:=(\mathrm{Bl}_{\mathscr{A}[2]}P)/H$ is a regular projective model of $X(d)=(\mathrm{Bl}_{A(d)[2]}A(d))/H$, and it is clear that the action of $\boldsymbol{\mu}_d$ on $X(d)$ extends to $\mathscr{Y}.$ Furthermore, we see (as in the proof of Theorem \ref{Kummersemistabletheorem}) that the reduced special fibre $(Y_k)_{\mathrm{red}}$ is a divisor with normal crossings on $\mathscr{Y}.$ In order to see that $\mathscr{Y}$ really is an equivariant Kulikov model, all we have to show is that $\omega_{\mathscr{Y}/\Og_{K(d)}}((\mathscr{Y}_k)_{\mathrm{red}})$ is trivial. Let $U_0:=\mathrm{Bl}_{\mathscr{A}[2]}\mathscr{A}\subseteq \mathrm{Bl}_{\mathscr{A}[2]}P$, and let $U_1:=P\backslash \mathscr{A}[2]$. Then $U_0/H,$ $U_1/H$ form an open cover of $\mathscr{Y}.$ Let $\pi\colon U_0\to \mathscr{A}$ be the morphism given by blowing up, and let $\rho\colon U_0\to U_0/H$ be the quotient map. As in the proof of Theorem \ref{Kummersemistabletheorem}, we can find no-where vanishing global 2-forms $\omega$ and $\beta$ on $\mathscr{A}$ and $U_0/H$, respectively, such that $\pi^\ast\omega=\rho^\ast\beta.$ By \cite{HN3}, Remark 5.1.7, we know that the line bundle $\omega_{P/\Og_K} ((P_k)_{\mathrm{red}})$ is trivial. If we choose a no-where vanishing global section $\eta$ of $\omega_{P/\Og_{K(d)}}((P_k)_{\mathrm{red}})$, we may assume, without loss of generality, that $\eta\mid_{\mathscr{A}}=\omega.$ 
We shall denote the morphism $\mathrm{Bl}_{\mathscr{A}[2]}P\to P$ also by $\pi$ by abuse of notation. Then the 2-form $\pi^\ast\eta$ is invariant under the action of $H$ on $\mathrm{Bl}_{\mathscr{A}[2]}P$, because $\pi^\ast\eta\mid_{\mathrm{Bl}_{\mathscr{A}[2]}\mathscr{A}}=\pi^\ast\omega,$ which is equal to $\rho^\ast\beta$, so it must be $H$-invariant. It follows in particular that $\pi^\ast\eta$ and $[-1]^\ast\pi^\ast\eta$ coincide on a dense open subscheme of $\mathrm{Bl}_{\mathscr{A}[2]}P$, so they must coincide everywhere. Hence $\pi^\ast\eta$ descends to a global section (which we shall also call $\beta$) of $\omega_{\mathscr{Y}/\Og_{K(d)}}((\mathscr{Y}_k)_{\mathrm{red}})$ (at this point we use that the map $U_1\to U_1/H$ is étale, that $U_0\to U_0/H$ is tamely ramified, and that $\pi^\ast \eta$ vanishes along the exceptional divisor). We already know from the arguments presented in the proof of Theorem \ref{Kummersemistabletheorem} that $\beta$ does not vanish on $U_0/H.$ It is clear that $\eta$ does not vanish on $U_1=P\backslash \mathscr{A}[2]$, so $\beta$ vanishes no-where. Hence $\omega_{\mathscr{Y}/\Og_{K(d)}}((\mathscr{Y}_k)_{\mathrm{red}})\cong \Og_{\mathscr{Y}}.$ Note that there exists $d_0\in \N$ which is minimal with the property that $A$ acquires semistable reduction and $A[2]$ becomes constant over $K(d_0).$ Hence the Theorem follows.
\end{proof}
\begin{corollary}
Let $X$ be a Kummer surface over $K$ (recall that the residue field $k$ of $\Og_K$ has characteristic 0). Then $X$ satisfies the monodromy property (\cite{HN3}, Definition 2.3.5). 
\end{corollary}
\begin{proof}
This follows from \cite{HN3}, Corollary 5.3.3 together with Theorem \ref{equivariantKulikovtheorem}.
\end{proof}
Using a similar method, we may also deduce the following Corollary, for which we shall only assume that the residue field $k$ of $\Og_K$ have characteristic different from 2 (in other words, $p=\mathrm{char}\, k>2$ as well as $\mathrm{char}\, k=0$ are allowed).
\begin{corollary}
Let $A$ be an Abelian surface over $K$ and let $X$ be the associated Kummer surface. Assume that $\Gal(\overline{K}/K)$ acts unipotently on $H^2_{\et}(X_{\overline{K}}, \Q_\ell)$ for some $\ell\not=p.$ Then $X$ admits a Kulikov model $\mathscr{X}\to \Spec \Og_K$ which is a scheme. Note that we do not make any assumptions about the reduction of $A$ over $\Og_K.$ In particular, $X$ admits a Kulikov model over $K$ which is a scheme as soon as $X$ has semistable reduction.
\end{corollary}
\begin{proof}
By Proposition \ref{quadratictwistproposition}, we may assume without loss of generality that $A$ has semiabelian reduction already over $K$. Furthermore, Proposition \ref{converseproposition} implies that $A[2]$ is a constant $K$-group scheme. (In both Propositions, although we assumed that $X$ have semistable reduction, the only fact we used in their proofs was that the Galois action on $H^2_{\et}(X_{\overline{K}}, \Q_\ell)$ is unipotent. Hence one sees easily that the conditions of both Propositions can be replaced by those of the Corollary). In the proof of the preceding Theorem, we had to extend the ground field in order to be able to make these assumptions on $A$, and the assumption that $k$ have characteristic 0 was only used during this first step, to ensure that all finite extensions of $K$ are of a particular form. Hence we may proceed as in the proof of the previous Theorem, using the trivial group instead of $\boldsymbol{\mu}_d.$
\end{proof}\\
$\mathbf{Remark.}$ In the preceding Corollary, we do not claim that the Kulikov model $\mathscr{X}\to \Spec \Og_K$ is a strict Kulikov model. In other words, the Corollary does not state that the special fibre of $\mathscr{X}\to \Spec \Og_K$ is reduced.\\
\\
$\mathbf{Acknowledgement.}$ The author is very grateful to Professor Alexei Skorobogatov for suggesting that he study degenerations of K3 surfaces, as well as for reading an earlier version of this paper and suggesting various changes which greatly improved the presentation of this paper. He would also like to express his gratitude to Dr. Johannes Nicaise for several helpful conversations, and particularly for pointing out to him that something along the lines of Corollary \ref{Pkulikovcorollary} might be true. This work was supported by the Engineering and Physical Sciences Research Council [EP/L015234/1], and the EPSRC Centre for Doctoral Training in Geometry and Number Theory (London School of Geometry and Number Theory), University College London.


\begin{thebibliography}{10}
\bibitem{B}
B\u adescu, L. S.
\textit{Algebraic Surfaces}. Universitext. Springer-Verlag, New York, 2001.

\bibitem{BLR}
Bosch, S., L\"utkebohmert, W., Raynaud, M.
\textit{Néron models}. Ergeb. Math. Grenzgeb. (3), Springer-Verlag, 1990.

\bibitem{Br}
Breuil, C. 
\textit{Groupes $p$-divisibles, groupes finis et modules filtrés}. Ann. of Math. 152 (2000), pp. 489-549.

\bibitem{CL}
Chiarellotto, B., Lazda, C. 
\textit{Combinatorial degenerations of surfaces and Calabi-Yau threefolds}. Algebra and Number Theory 10(10), 2016.

\bibitem{CI}
Coleman, R., Iovita, A.
\textit{The Frobenius and monodromy operator for curves and Abelian varieties}. Duke Math. J., Vol. 97, No. 1, 1999.

\bibitem{DF}
DeMeyer, F. R., Ford, T. J. 
\textit{On the Brauer group of Surfaces}. J. Algebra 86, pp. 259-271, 1984.

\bibitem{FC}
Faltings, G., Chai, C.-L.
\textit{Degeneration of Abelian varieties}. Ergeb. Math. Grenzgeb. (3), Springer-Verlag, 1990.

\bibitem{Fo}
Fontaine, J.-M.
\textit{Le corps des périodes $p$-adiques}. Périodes $p$-adiques (Séminaire de Bures, 1988), Asterisque, vol. 223, 1994, pp. 59-101.

\bibitem{EGAIII}
Grothendieck, A.
\textit{Elements du géométrie algébrique: III. Étude cohomologique des faisceaux cohérents, Première partie}. Publ. Math. IHES, tome 11 (1961), pp. 5-167.

\bibitem{SGA7}
Grothendieck, A., Deligne, P., Katz, N. with Raynaud, M. and Rim, D. S.
\textit{Groupes de monodromie en géométrie algébrique}. Lecture Notes in Math. 269, 270, 305 (1972-73).

\bibitem{Ha}
Hartshorne, R.
\textit{Stable reflexive sheaves}. Math. Ann. 254, pp. 121-176, 1980.

\bibitem{HN3}
Halle, L. H., Nicaise, J.
\textit{Motivic zeta functions of degenerating Calabi-Yau varieties}. Math. Ann. 370, pp. 1277-1320, 2018.

\bibitem{Huy}
Huybrechts, D.
\textit{Lectures on K3 surfaces}. Cambridge Studies in Advanced Mathematics (158), Cambridge University Press, 2016.

\bibitem{Ito}
Ito, K.
\textit{Unconditional construction of K3 surfaces over finite fields with given $L$-function in large characteristic}. Preprint; available at https://arxiv.org/pdf/1612.05382.pdf.

\bibitem{Ka}
Kato, K. 
\textit{Semistable reduction and $p$-adic étale cohomology}. Périodes $p$-adiques (Séminaire de Bures, 1988), Asterisque, vol. 223, 1994, pp. 269-293.

\bibitem{Kl}
Kleiman, S.
\textit{Algebraic cycles and the Weil conjectures}. Dix exposes sur la cohomologie des schémas. Grothendieck, A., Kuiper, N. H. eds., North-Holland Publishing Company, 1986.

\bibitem{Ku}
Kulikov, V. S. 
\textit{Degenerations of K3 surfaces and Enriques surfaces}. Izv. Akad. Nauk SSSR Ser. Mat. 41 (1977), no. 5, pp.1008-1042

\bibitem{Kü}
K\"unnemann, K.
\textit{Projective regular models for Abelian varieties, semistable reduction, and the height pairing}. Duke Math. J., Vol. 95, No. 1, 1998.

\bibitem{LM}
Liedtke, C., Matsumoto, Y.
\textit{Good reduction of K3 surfaces}. Compos. Math. 154 (2018), pp. 1-35.

\bibitem{L}
Loerke, K.
\textit{Reduction of Abelian varieties and Grothendieck's pairing}. Preprint, 2009.

\bibitem{Ma}
Matsumoto, Y.
\textit{On good reduction of some K3 surfaces related to Abelian surfaces}. Tohoku Math. J. 67 (2015), pp. 83-104

\bibitem{Ma2}
Matsumoto, Y.
\textit{A good reduction criterion for K3 surfaces}. Math. Z. (2015), Vol. 279, pp. 241-266.

\bibitem{Mau}
Maulik, D.
\textit{Supersingular K3 surfaces for large primes}. Duke Math. J., Vol. 163, No. 13, 2014.

\bibitem{M}
Mumford, D. 
\textit{An analytic construction of degenerating Abelian varieties over complete rings}. Compos. Math. 24, Fasc. 3, 1972, pp. 239-272. 

\bibitem{N}
Nakayama, C.
\textit{Nearby cycles for log smooth families}. Compos. Math. 112, 1998, pp. 45-75.

\bibitem{Per}
Persson, U.
\textit{On degenerations of algebraic surfaces}. Mem. Amer. Math. Soc. 11 (1977), no. 189. 

\bibitem{PerPink}
Persson, U., Pinkham, H.
\textit{Degenerations of surfacey witj trivial canonical bundle}. Ann. of Math. 113 (1981), pp. 45-66. 

\bibitem{RZ}
Rapoport, M., Zink, T. 
\textit{Über die lokale Zetafunktion von Shimuravarietäten. Monodromiefiltration und verschwindende Zyklen in ungleicher Charakteristik}. Invent. Math. 68, pp. 21-101, 1982.

\bibitem{Ra}
Raynaud, M.
\textit{Faisceaux amples sur les schémas en groupes at les espaces homogènes}. Lecture Notes in Math. 119, Springer-Verlag, 1970.

\bibitem{S}
Saito, T. 
\textit{Log smooth extension of a family of curves and semi-stable reduction}. J. Alg. Geom. 13 (2004), pp. 287-321.

\bibitem{SkZ}
Skorobogatov, A., Zarhin, Y.
\textit{Kummer varieties and their Brauer groups}. To appear in Pure Appl. Math. Quart.

\bibitem{Ts}
Tsuji, T.
\textit{$p$-adic étale and crystalline cohomology in the semi-stable reduction case}. Invent. Math. 137, pp. 233-411, 1999.
\end{thebibliography}
\end{document}